\documentclass[reqno,a4paper]{amsart}

\usepackage{amsthm}
\usepackage[leqno]{amsmath}
\usepackage[all]{xy} \SelectTips{eu}{}
\usepackage{latexsym,amsfonts,amssymb}
\usepackage{hyperref}
\usepackage{mathptmx}
\usepackage{nicefrac}
\usepackage{array}
\usepackage{xcolor}

\newcommand{\numberseries}{\bfseries}   

\newlength{\thmtopspace}                
\newlength{\thmbotspace}                
\newlength{\thmheadspace}               
\newlength{\thmindent}                  

\newtheoremstyle{fixed bf head,slanted body}
                {\thmtopspace}{\thmbotspace}{\slshape}
                {\thmindent}{\bfseries}{.}{\thmheadspace}
                {{\numberseries \thmnumber{#2\;}}\thmname{#1}\thmnote{ (#3)}}

\newtheoremstyle{variable bf head,slanted body}
                {\thmtopspace}{\thmbotspace}{\slshape}
                {\thmindent}{\bfseries}{.}{\thmheadspace}
                {{\numberseries \thmnumber{#2\;}}\thmname{#1}\thmnote{ #3}}

\newtheoremstyle{fixed bf head,upright body}
                {\thmtopspace}{\thmbotspace}{\upshape}
                {\thmindent}{\bfseries}{.}{\thmheadspace}
                {{\numberseries \thmnumber{#2\;}}\thmname{#1}\thmnote{ (#3)}}

\newtheoremstyle{numbered paragraph}
                {\thmtopspace}{\thmbotspace}{\upshape}
                {\thmindent}{\upshape}{}{\thmheadspace}
                {{\numberseries \thmnumber{#2.}}}

\theoremstyle{fixed bf head,slanted body}
\newtheorem{res}{}[section]
\newtheorem{thm}[res]{Theorem}          \newtheorem*{thm*}{Theorem}
\newtheorem{prp}[res]{Proposition}      \newtheorem*{prp*}{Proposition}
        \newtheorem*{cor*}{Corollary}
\newtheorem{lem}[res]{Lemma}            \newtheorem*{lem*}{Lemma}

\theoremstyle{variable bf head,slanted body}
     \newtheorem*{introthm*}{Theorem}
   \newtheorem*{introcor*}{Corollary}

\theoremstyle{fixed bf head,upright body}
            \newtheorem*{stp*}{Setup}
\newtheorem{dfn}[res]{Definition}       \newtheorem*{dfn*}{Definition}
     \newtheorem*{con*}{Construction}
\newtheorem{obs}[res]{Observation}      \newtheorem*{obs*}{Observation}
\newtheorem{rmk}[res]{Remark}           \newtheorem*{rmk*}{Remark}
\newtheorem{exa}[res]{Example}          \newtheorem*{exa*}{Example}
         \newtheorem*{qst*}{Question}

\theoremstyle{numbered paragraph}
\newtheorem{ipg}[res]{}

\newlength{\thmlistleft}        
\newlength{\thmlistright}       
\newlength{\thmlistpartopsep}   
\newlength{\thmlisttopsep}      
\newlength{\thmlistparsep}      
\newlength{\thmlistitemsep}     

\newcounter{eqc} 
  {\end{list}}%

\newcounter{prt}
\newenvironment{prt}{\begin{list}{\upshape (\alph{prt})}%
    {\usecounter{prt}%
      \setlength{\leftmargin}{\thmlistleft}%
      \setlength{\labelwidth}{\thmlistleft}%
      \setlength{\rightmargin}{\thmlistright}%
      \setlength{\partopsep}{\thmlistpartopsep}%
      \setlength{\topsep}{\thmlisttopsep}%
      \setlength{\parsep}{\thmlistparsep}%
      \setlength{\itemsep}{\thmlistitemsep}}}%
  {\end{list}}%

  \newcommand{\proofoftag}[2][:]{(#2)#1}

\newcounter{rqm}
  {\end{list}}%

\newenvironment{itemlist}{\nopagebreak \begin{list}{$\bullet$}%
    {\setlength{\leftmargin}{\thmlistleft}%
      \setlength{\labelwidth}{\thmlistleft}%
      \setlength{\rightmargin}{\thmlistright}%
      \setlength{\partopsep}{\thmlistpartopsep}%
      \setlength{\topsep}{\thmlisttopsep}%
      \setlength{\parsep}{\thmlistparsep}%
      \setlength{\itemsep}{\thmlistitemsep}}}%
  {\end{list}}%

\newcommand{\pgref}[1]{\ref{#1}}
\newcommand{\thmref}[2][Theorem~]{#1\pgref{thm:#2}}

\newcommand{\prpref}[2][Proposition~]{#1\pgref{prp:#2}}
\newcommand{\lemref}[2][Lemma~]{#1\pgref{lem:#2}}
\newcommand{\obsref}[2][Observation~]{#1\pgref{obs:#2}}

\newcommand{\dfnref}[2][Definition~]{#1\pgref{dfn:#2}}
\newcommand{\exaref}[2][Example~]{#1\pgref{exa:#2}}
\newcommand{\rmkref}[2][Remark~]{#1\pgref{rmk:#2}}
\newcommand{\secref}[2][Section~]{#1\ref{sec:#2}}
\newcommand{\appref}[2][Appendix~]{#1\ref{app:#2}}

\renewcommand{\eqref}[1]{(\pgref{eq:#1})}

\makeatletter
\def\@nobreak@#1{\mathchoice%
  {\nobreakdef@\displaystyle\f@size{#1}}%
  {\nobreakdef@\nobreakstyle\tf@size{\firstchoice@false #1}}%
  {\nobreakdef@\nobreakstyle\sf@size{\firstchoice@false #1}}%
  {\nobreakdef@\nobreakstyle\ssf@size{\firstchoice@false #1}}%
  \check@mathfonts}%
\def\nobreakdef@#1#2#3{\hbox{{%
                    \everymath{#1}%
                    \let\f@size#2\selectfont%
                    #3}}}%
\makeatother

\DeclareFontFamily{T1}{cmex}{}
\DeclareFontShape{T1}{cmex}{m}{n}{<-> s * [0.89] cmex10}{}
\DeclareSymbolFont{cmlargesymbols}{T1}{cmex}{m}{n}
\DeclareMathSymbol{\mycoprod}{\mathop}{cmlargesymbols}{"60} 
\DeclareMathSymbol{\myprod}{\mathop}{cmlargesymbols}{"51} \let\prod\myprod

\DeclareSymbolFont{usualmathcal}{OMS}{cmsy}{m}{n}
\DeclareSymbolFontAlphabet{\mathcal}{usualmathcal}

\DeclareSymbolFont{letters}{OML}{txmi}{m}{it}

\DeclareMathSymbol{\alpha}{\mathord}{letters}{"0B}
\DeclareMathSymbol{\beta}{\mathord}{letters}{"0C}
\DeclareMathSymbol{\gamma}{\mathord}{letters}{"0D}
\DeclareMathSymbol{\delta}{\mathord}{letters}{"0E}
\DeclareMathSymbol{\epsilon}{\mathord}{letters}{"0F}
\DeclareMathSymbol{\zeta}{\mathord}{letters}{"10}
\DeclareMathSymbol{\eta}{\mathord}{letters}{"11}
\DeclareMathSymbol{\theta}{\mathord}{letters}{"12}
\DeclareMathSymbol{\iota}{\mathord}{letters}{"13}
\DeclareMathSymbol{\kappa}{\mathord}{letters}{"14}
\DeclareMathSymbol{\lambda}{\mathord}{letters}{"15}
\DeclareMathSymbol{\mu}{\mathord}{letters}{"16}
\DeclareMathSymbol{\nu}{\mathord}{letters}{"17}
\DeclareMathSymbol{\xi}{\mathord}{letters}{"18}
\DeclareMathSymbol{\pi}{\mathord}{letters}{"19}
\DeclareMathSymbol{\rho}{\mathord}{letters}{"1A}
\DeclareMathSymbol{\sigma}{\mathord}{letters}{"1B}
\DeclareMathSymbol{\tau}{\mathord}{letters}{"1C}
\DeclareMathSymbol{\upsilon}{\mathord}{letters}{"1D}
\DeclareMathSymbol{\phi}{\mathord}{letters}{"1E}
\DeclareMathSymbol{\chi}{\mathord}{letters}{"1F}
\DeclareMathSymbol{\psi}{\mathord}{letters}{"20}
\DeclareMathSymbol{\omega}{\mathord}{letters}{"21}
\DeclareMathSymbol{\varepsilon}{\mathord}{letters}{"22}
\DeclareMathSymbol{\vartheta}{\mathord}{letters}{"23}
\DeclareMathSymbol{\varpi}{\mathord}{letters}{"24}
\DeclareMathSymbol{\varrho}{\mathord}{letters}{"25}
\DeclareMathSymbol{\varsigma}{\mathord}{letters}{"26}
\DeclareMathSymbol{\varphi}{\mathord}{letters}{"27}

\DeclareMathSymbol{\Gamma}{\mathord}{letters}{"00}
\DeclareMathSymbol{\Delta}{\mathord}{letters}{"01}
\DeclareMathSymbol{\Theta}{\mathord}{letters}{"02}
\DeclareMathSymbol{\Lambda}{\mathord}{letters}{"03}
\DeclareMathSymbol{\Xi}{\mathord}{letters}{"04}
\DeclareMathSymbol{\Pi}{\mathord}{letters}{"05}
\DeclareMathSymbol{\Sigma}{\mathord}{letters}{"06}
\DeclareMathSymbol{\Upsilon}{\mathord}{letters}{"07}
\DeclareMathSymbol{\upPhi}{\mathord}{letters}{"08}
\DeclareMathSymbol{\Psi}{\mathord}{letters}{"09}
\DeclareMathSymbol{\Omega}{\mathord}{letters}{"0A}

\DeclareMathSymbol{\upGamma}{\mathalpha}{operators}{"00}
\DeclareMathSymbol{\upDelta}{\mathalpha}{operators}{"01}
\DeclareMathSymbol{\upTheta}{\mathalpha}{operators}{"02}
\DeclareMathSymbol{\upLambda}{\mathalpha}{operators}{"03}
\DeclareMathSymbol{\upXi}{\mathalpha}{operators}{"04}
\DeclareMathSymbol{\upPi}{\mathalpha}{operators}{"05}
\DeclareMathSymbol{\upSigma}{\mathalpha}{operators}{"06}
\DeclareMathSymbol{\upUpsilon}{\mathalpha}{operators}{"07}
\DeclareMathSymbol{\upPhi}{\mathalpha}{operators}{"08}
\DeclareMathSymbol{\upPsi}{\mathalpha}{operators}{"09}
\DeclareMathSymbol{\upOmega}{\mathalpha}{operators}{"0A}

\renewcommand{\c}[1]{\mathcal{#1}}

\newcommand{\ira}[2][\cong]{\mspace{4mu}\smash{\stackrel{\text{\raisebox{2.3pt}{$#1$}}}{\smash{#2}}}\mspace{4mu}}

\newcommand{\ASpec}[1]{\mathrm{ASpec}\mspace{2mu}#1}
\newcommand{\ASpecp}[1]{\mathrm{ASpec}(#1)}
\newcommand{\ASupp}[1]{\mathrm{ASupp}\mspace{2mu}#1}
\newcommand{\Spec}[1]{\mathrm{Spec}\mspace{3mu}#1}
\renewcommand{\a}[1]{\langle#1\rangle}
\renewcommand{\aa}[1]{\left\langle#1\right\rangle}
\newcommand{\Mod}[1]{#1\textnormal{-Mod}}
\newcommand{\fgMod}[1]{#1\textnormal{-mod}}
\newcommand{\Ch}[1]{\mathrm{Ch}\mspace{2mu}#1}
\newcommand{\com}[2]{(#1 \!\downarrow\! #2)}
\newcommand{\Rep}[2]{\operatorname{Rep}(#1,#2)}
\newcommand{\Hom}[3]{\operatorname{Hom}_{#1}(#2,#3)}
\newcommand{\Ker}[1]{\operatorname{Ker}#1}
\renewcommand{\Im}[1]{\operatorname{Im}#1}
\newcommand{\EssIm}[1]{\operatorname{Ess.Im}#1}

\begin{document}

\title{Computations of atom spectra}

\author{Rune Harder Bak \ }
\address{(R.H.B.) Department of Mathematical Sciences, Universitetsparken 5, University of Co\-penhagen, 2100 Copenhagen {\O}, Denmark} 
\email{bak@math.ku.dk}

\author{ \ Henrik Holm}
\address{(H.H.) Department of Mathematical Sciences, Universitetsparken 5, University of Co\-penhagen, 2100 Copenhagen {\O}, Denmark} 
\email{holm@math.ku.dk}
\urladdr{http://www.math.ku.dk/\~{}holm/}


\keywords{Atom spectrum; comma category; quiver with relations; representation of quiver.}

\subjclass[2010]{16G20; 18E10}



\begin{abstract}
  This is a contribution to the theory of atoms in abelian categories recently developed in a series of papers by Kanda. We present a method that enables one to explicitly compute the atom spectrum of the module category over a wide range of non-commutative rings. We illustrate our method and results by several examples.
\end{abstract}

\maketitle


\section{Introduction}
\label{sec:Introduction}

Building on works of Storrer \cite{Sto72}, Kanda has, in a recent series of papers \cite{Kanda-Serre, Kanda-Extension, Kanda-Specialization-orders}, developed the theory of atoms in abelian categories. The fundamental idea is to assign to every abelian category $\c{A}$ the \emph{atom spectrum}, denoted by $\ASpec{\c{A}}$, in such a way that when $\Bbbk$ is a commutative ring, then $\ASpecp{\Mod{\Bbbk}}$ recovers the prime ideal spectrum $\Spec{\Bbbk}$. In \secref{Kanda} we recall a few basic definitions and facts from Kanda's theory.

Strong evidence suggests that Kanda's atom spectrum really is the ``correct'', and a very interesting, generalization of the prime ideal spectrum to abstract abelian categories. For example, in \cite[Thm.~5.9]{Kanda-Serre} it is proved that for any locally noetherian Grothendieck category $\c{A}$ there is a bijective correspondance between $\ASpec{\c{A}}$ and isomorphism
classes of indecomposable injective objects in $\c{A}$. This is a generalization of Matlis' bijective~corre\-spondance between $\Spec{\Bbbk}$ and the set of isomorphism classes of indecomposable injective $\Bbbk$-modules over a commutative noetherian ring $\Bbbk$; see \cite{EMt58}. Furthermore, in \cite[~Thm.~5.5]{Kanda-Serre} it is shown that there are bijective correspondances between open subsets of $\ASpec{\c{A}}$, Serre subcategories of $\mathrm{noeth}\mspace{2mu}\c{A}$, and localizing subcategories of $\c{A}$. This generalizes Ga\-bri\-el's bijective correspondances \cite{PGb62} between specialization-closed subsets of $\Spec{\Bbbk}$, Serre subcategories of $\fgMod{\Bbbk}$, and localizing subcategories of $\Mod{\Bbbk}$ for a commutative noetherian ring $\Bbbk$. From a theoretical viewpoint, these results are very appealing, however, in the literature it seems that little effort has been put into actually computing the atom spectrum.

In this paper, we add value to the results mentioned above, and to other related results, by explicitly computing/describing the atom spectrum---not just as a set, but as a partially ordered set and as a topological space---of a wide range of abelian \mbox{categories}. Our main technical result, \thmref{bijective}, shows that if \mbox{$F_{\mspace{-1mu}i} \colon \c{A}_i \to \c{B}$} $(i \in I)$ is a family of functors between abelian categories satisfying suitable assumptions, then there is a homeomorphism and an order-isomorphism $f \colon \textstyle\bigsqcup_{i \in I}\,\ASpec{\c{A}_i} \to \ASpec{\c{B}}$. Hence, if all the atom spectra $\ASpec{\c{A}_i}$ are known, then so is $\ASpec{\c{B}}$. One special case of this result is:

\begin{introthm*}[A]
    Let $(Q,\c{R})$ be a quiver with admissible relations and finitely many~ver\-ti\-ces. Let $\Bbbk Q$ be the path algebra of $Q$ and consider the two-sided ideal $I=(\c{R})$ in $\Bbbk Q$ generated by $\c{R}$. There is an injective, continuous, open, and order-preserving map,
\begin{displaymath}  
  \tilde{f} \colon \textstyle\bigsqcup_{i \in Q_0} \Spec{\Bbbk} \longrightarrow \ASpecp{\Mod{\Bbbk Q/I}}\;,
\end{displaymath}  
given by $\textnormal{($i^\mathrm{th}$ copy of $\Spec{\Bbbk}$) $\ni \mathfrak{p}$} 
  \,\longmapsto \a{\mspace{2mu}\Bbbk Q/\mspace{2mu}\tilde{\mathfrak{p}}(i)\mspace{1mu}}$. If, in addition, $(Q,\c{R})$ is right rooted, then $\tilde{f}$ is also surjective, and hence it is a homeomorphism and an order-isomorphism.
\end{introthm*}

We prove Theorem~A in \secref{Application-quiver}, where we also give the definitions of admisible~relations (\dfnref[]{admissible}), right-rooted quivers (\dfnref[]{right-rooted}), and of the ideals $\tilde{\mathfrak{p}}(i)$ (\dfnref[]{p-tilde}). In the terminology of Kanda \cite[Def.~6.1]{Kanda-Serre}, Theorem~A yields that each $\tilde{\mathfrak{p}}(i)/I$, where $\mathfrak{p}$ is a prime ideal in $\Bbbk$ and $i$ is a vertex in $Q$, is a \emph{comonoform} left ideal in the ring $\Bbbk Q/I$ (and, in the case where $(Q,\c{R})$ is right rooted, these comonoform ideals represent all the atoms of $\Mod{\Bbbk Q/I}$).

Theorem~A applies e.g.~to show that for every $n,m \geqslant 1$ the map
\begin{displaymath}  
  \Spec{\Bbbk} \longrightarrow \ASpecp{\Mod{\Bbbk\langle x_1,\ldots,x_n \rangle/(x_1,\ldots,x_n)^m}}
  \quad \text{ given by } \quad \mathfrak{p} \longmapsto \a{\Bbbk/\mathfrak{p}}
\end{displaymath} 
is a homeomorphism and an order-isomorphism; see \exaref{free-algebra}. Actually, Theorem~A is a special case of \thmref{quiver} which yields a homeomorphism and an order-isomorphism $\ASpecp{\Rep{(Q,\c{R})}{\c{A}}} \cong \bigsqcup_{i \in Q_0} \ASpec{\c{A}}$ for every right rooted quiver $(Q,\c{R})$ with admissible relations ($Q$ may have infinitely many vertices) and any $\Bbbk$-linear abelian category $\c{A}$. From this stronger result one gets e.g.~$\ASpecp{\Ch{\c{A}}} \cong \bigsqcup_{i \in \mathbb{Z}}\, \ASpec{\c{A}}$; see \exaref{Ch}.

In \secref{Comma} we apply the previously mentioned (technical/abstract) \thmref{bijective} to compute the atom spectrum of comma categories:

\begin{introthm*}[B]
  Let \smash{$\c{A} \xrightarrow{\ U \ } \c{C} \xleftarrow{\ V \ } \c{B}$} be functors between abelian categories, where $U$ has a right adjoint and $V$ is left exact. Let $\com{U}{V}$ be the associated comma category. There is a homeomorphism and an order-isomorphism,
\begin{displaymath}  
  f \colon \ASpec{\c{A}} \,\sqcup\, \ASpec{\c{B}} \stackrel{\sim}{\longrightarrow} \ASpec{\com{U}{V}}\;,
\end{displaymath}  
given by $\a{H} \longmapsto \a{S_{\mspace{-5mu}\c{A}}\mspace{1mu}H}$ for $\a{H} \in \ASpec{\c{A}}$ and $\a{H} \longmapsto \a{S_{\mspace{-4mu}\c{B}}\mspace{1mu}H}$ for $\a{H} \in \ASpec{\c{B}}$.
\end{introthm*}

Theorem~B applies e.g.~to show that for the non-commutative ring
\begin{displaymath}
  T = 
  \begin{pmatrix}
  A & 0 \\
  M & B
  \end{pmatrix},
\end{displaymath}
where $A$ and $B$ are commutative rings and $M = {}_BM_A$ is a $(B,A)$-bimodule, there is a homeomorphism and an order-isomorphism $\Spec{A} \,\sqcup\, \Spec{B} \longrightarrow \ASpecp{\Mod{T}}$ given by
\begin{displaymath}
  \Spec{A} \ni \mathfrak{p} \,\longmapsto\, \aa
  {T\!\!\bigg/\!\!
  \begin{pmatrix}
    \mathfrak{p} & 0 \\
    M & B
  \end{pmatrix}}
  \quad \textnormal{ and } \quad
  \Spec{B} \ni \mathfrak{q} \,\longmapsto\, \aa{T\!\!\bigg/\!\!
  \begin{pmatrix}
    A & 0 \\
    M & \mathfrak{q}
  \end{pmatrix}};
\end{displaymath}  
see~\exaref{T-Mod} for details.

We end the paper with \appref{Quiver} where we present some background material on representations of quivers with relations that is needed, and taken for granted, in \secref{Application-quiver}.

\section{Kanda's theory of atoms}
\label{sec:Kanda}

We recall a few definitions and results from Kanda's theory \cite{Kanda-Serre, Kanda-Extension, Kanda-Specialization-orders} of atoms.

\begin{ipg}
  \label{Aspec}
  Let $\c{A}$ be an abelian category. An object $H \in \c{A}$ is called \emph{monoform} if $H \neq 0$ and for every non-zero subobject $N \rightarrowtail H$ there exists no common non-zero subobject of $H$~and~$H/N$, i.e.~if there exist monomorphisms \mbox{$H \leftarrowtail X \rightarrowtail H/N$} in $\c{A}$, then $X=0$. See \cite[Def.~2.1]{Kanda-Serre}.
  
Two monoform objects $H$ and $H'$ in $\c{A}$ are said to be \emph{atom equivalent} if there exists a common non-zero subobject of $H$ and $H'$. Atom equivalence is an equivalence relation on the collection of monoform objects; the equivalence class of a monoform object $H$ is denoted by $\a{H}$ and is called an \emph{atom} in $\c{A}$. The collection of all atoms in $\c{A}$ is called the \emph{atom spectrum} of $\c{A}$ and denoted by $\ASpec{\c{A}}$. See \cite[Def.~2.7, Prop.~2.8, and Def.~3.1]{Kanda-Serre}.
\end{ipg}

The atom spectrum of an abelian category comes equipped with a topology and a partial order which we now explain.

\begin{ipg}
  \label{topology}
  The \emph{atom support} of an object $M \in \c{A}$ is defined in \cite[Def.~3.2]{Kanda-Serre} and is given by
  \begin{displaymath}
    \ASupp{M} = \left\{ \a{H} \in \ASpec{\c{A}} \ \left|\!\!
      \begin{array}{l}
       \textnormal{$H$ is a monoform object such that} \\
       \textnormal{$H \cong L/L'$ for some $L' \subseteq L \subseteq M$}
    \end{array}
    \right.\!\!\!
    \right\}.
  \end{displaymath}

A subset $\upPhi \subseteq \ASpec{\c{A}}$ is said to be \emph{open} if for every $\a{H} \in \upPhi$ there exists $H' \in \a{H}$ such that $\ASupp{H'} \subseteq \upPhi$. The collection of open subsets defines a topology, called the~\emph{localization topology}, on $\ASpec{\c{A}}$, see \cite[Def.~3.7 and Prop.~3.8]{Kanda-Serre}, and the collection
\begin{displaymath}
  \{ \ASupp{M} \ |\, M \in \c{A} \}
\end{displaymath}
is an open basis of this topology; see \cite[Prop.~3.2]{Kanda-Specialization-orders}.
\end{ipg}

\begin{ipg}
  \label{order}
  The topological space $\ASpec{\c{A}}$ is a so-called \emph{Kolmogorov space} (or a \emph{$T_0$-space}), see \cite[Prop.~3.5]{Kanda-Specialization-orders}, and any such space $X$ can be equipped with a canonical partial order $\leqslant$, called the \emph{specialization order}, where $x \leqslant y$ means that $x \in \overline{\{y\}}$ (the closure of $\{y\}$ in $X$). This partial order on $\ASpec{\c{A}}$ is more explicitly described in \cite[Prop.~4.2]{Kanda-Specialization-orders}.
\end{ipg}

\begin{lem}
  \label{lem:order-preserving}
  Let $X$ and $Y$ be Kolmogorov spaces equipped with their specialization orders. Any continuous map $f \colon X \to Y$ is automatically order-preserving.
\end{lem}

\begin{proof}
  Assume that $x \leqslant y$ in $X$, that is, $x \in \overline{\{y\}}$. Then $f(x) \in f(\overline{\{y\}}) \subseteq \overline{f(\{y\})} = \overline{\{f(y)\}} $, where the inclusion holds as $f$ is continuous, and thus $f(x) \leqslant f(y)$ in $Y$.
\end{proof}

\begin{ipg}
  \label{Spec}
  For a commutative ring $\Bbbk$, its prime ideal spectrum coincides with the 
  atom spec\-trum of its module category in the following sense: By \cite[Props.~6.2, 7.1, and 7.2(1)]{Kanda-Serre}, see also \cite[p.~631]{Sto72}, there is a bijection of sets:
  \begin{displaymath}
    q \colon \Spec{\Bbbk} \longrightarrow \ASpecp{\Mod{\Bbbk}}
    \quad \text{given by} \quad
    \mathfrak{p} \longmapsto \a{\Bbbk/\mathfrak{p}}\;.
  \end{displaymath}
  This bijection is even an order-isomorphism between the partially ordered set \mbox{$(\Spec{\Bbbk},\subseteq)$} and $\ASpecp{\Mod{\Bbbk}}$ equipped with its specialization order; see \cite[Prop.~4.3]{Kanda-Specialization-orders}. Via $q$ the open subsets of $\ASpecp{\Mod{\Bbbk}}$ correspond to the \emph{speciali\-za\-tion-closed} subsets of $\Spec{\Bbbk}$; see \cite[Prop.~7.2(2)]{Kanda-Serre}. In this paper, we always consider $\Spec{\Bbbk}$ as a partially ordered set w.r.t.~to inclusion and as a topological space in which the open sets are the specialization-closed ones. In this way, the map $q$ above is an order-isomorphism and a homeomorphism.\footnote{\,We emphasize that $q$ is not a homeomorphism when $\Spec{\Bbbk}$ is equipped with the (usual) Zariski topology! In the case where $\Bbbk$ is noetherian, the topological space $\Spec{\Bbbk}$ considered by us and Kanda \cite{Kanda-Serre} is the \emph{Hochster dual}, in the sense of \cite[Prop.~8]{Hochster}, of the spectral space $\Spec{\Bbbk}$ with Zariski topology.}  
\end{ipg}  

\section{The main result}

In this section, we explain how a suitably nice functor $F \colon \c{A} \to \c{B}$ between abelian categories induces a map $\ASpec{F} \colon \ASpec{\c{A}} \to \ASpec{\c{B}}$. The terminology in the following definition is inspired by a similar terminology from 
Diers~\cite[Chap.~1.8]{Diers}, where it is defined what it means for a functor to lift direct factors.

\begin{dfn}
  \label{dfn:cs}
  Let $F \colon \c{A} \to \c{B}$ be a functor. We say that $F$~\emph{lifts subobjects} if for any $A \in \c{A}$ and any monomorphism $\iota \colon B \rightarrowtail FA$ in $\c{B}$ there exist a monomorphism $\iota' \colon A' \rightarrowtail A$ in $\c{A}$ and an isomorphism $B \ira{\longrightarrow} FA'$ such that the following diagram commutes:
\begin{displaymath}
  \xymatrix@!=0.4pc{
    B \ar[dr]_-{\cong} \ar@{>->}[rr]^-{\iota} & & FA
    \\
    {} & {}\phantom{.}FA'\mspace{2mu}.\mspace{-2mu} \ar[ur]_-{F\iota'} & {}
  }
\end{displaymath}  
(We will usually suppress the isomorphism and treat it as an equality $B=FA'$.)
\end{dfn}

\begin{rmk}
  \label{rmk:inj-on-obj}
  Recall that any full and faithful (= fully faithful) functor $F \colon \c{A} \to \c{B}$ is injective on objects up to isomorphism, that is, if $FA \cong FA'$ in $\c{B}$, then $A \cong A'$ in $\c{A}$.
\end{rmk}

\begin{obs}
  \label{obs:inclusion}
  Let $M$ be an object in $\c{A}$. If $\a{H} \in \ASupp{M}$ then, by definition, one has $\a{H}=\a{H'}$ where $H'$ is a monoform object of the form $H' \cong L/L'$ for some $L' \subseteq L \subseteq M$. Now \cite[Prop.~3.3]{Kanda-Serre} applied to $0 \to L' \to L \to H' \to 0$ and $0 \to L \to M \to M/L \to 0$ yield inclusions \mbox{$\ASupp{H'} \subseteq \ASupp{L} \subseteq \ASupp{M}$}.  
\end{obs}

\begin{prp}
  \label{prp:ASpecF}
  Let \mbox{$F \colon \c{A} \to \c{B}$} be a full, faithful, and exact functor between abelian categories that lifts subobjects. There is a well-defined map,
\begin{displaymath}  
  \ASpec{F} \colon \ASpec{\c{A}} \longrightarrow \ASpec{\c{B}}
  \quad \text{ given by } \quad 
  \a{H} \longmapsto \a{FH}\;,
\end{displaymath}   
which is injective, continuous, open, and order-preserving.  
\end{prp}

\begin{proof}
  First we argue that for any object $H \in \c{A}$ we have:
  \begin{equation}
    \label{eq:H-FH-monoform}
    \textnormal{$H$ is monoform (in $\c{A}$)}
    \ \iff \ 
    \textnormal{$FH$ is monoform (in $\c{B}$)}\;.    
  \end{equation}
  
  ``$\Leftarrow$'': Assume that $FH$ is monoform. By definition, $FH$ is non-zero, so $H$ must be non-zero as well. Let $M$ be a non-zero subobject of $H$ and assume that there are monomorphisms \mbox{$H \leftarrowtail X \rightarrowtail H/M$}. We must prove that $X = 0$. As $F$ is exact we get monomorphisms \mbox{$FH \leftarrowtail FX \rightarrowtail F(H/M) \cong (FH)/(FM)$}. Note that $FM \neq 0$ by \rmkref{inj-on-obj}, so it follows that $FX=0$ since $FH$ is monoform. Hence $X=0$, as desired.
  
  ``$\Rightarrow$'': Assume that $H$ is monoform. As $H \neq 0$ we have $FH \neq 0$ by \rmkref{inj-on-obj}. Let $N$ be a non-zero subobject of $FH$ and let \mbox{$FH \leftarrowtail Y \rightarrowtail (FH)/N$} be monomorphisms. We must prove that $Y = 0$. Since $F$ lifts subobjects, the monomorphism $N \rightarrowtail FH$ is the image under $F$ of a monomorphisms $M \rightarrowtail H$. As $FM=N$ is non-zero, so is $M$. Since $F$ is exact, the canonical morphism $(FH)/N = (FH)/(FM) \to F(H/M)$ is an isomorphism. By precomposing this isomorphism with $Y \rightarrowtail (FH)/N$ we get a monomorphism $Y \rightarrowtail F(H/M)$, which is then the image under $F$ of some monomorphism $X \rightarrowtail H/M$. The monomorphism $FH \leftarrowtail Y$ is also the image of a monomorphism $H \leftarrowtail X'$, and since $FX = Y = FX'$ we have $X \cong X'$ by \rmkref{inj-on-obj}. Hence there are monomorphisms \mbox{$H \leftarrowtail X \rightarrowtail H/M$}, and as $H$ is monoform we conclude that $X=0$. Hence $Y=FX=0$, as desired.
 
Next note that if $H$ and $H'$ are monoform objects in $\c{A}$ which are atom equivalent, i.e. they have a common non-zero subobject $M$, then $FM$ is a common non-zero subobject of $FH$ and $FH'$, and hence $FH$ and $FH'$ are atom equivalent monoform objects in $\c{B}$. This, together with the implication ``$\Rightarrow$'' in \eqref{H-FH-monoform}, shows that the map $\ASpec{F}$ is well-defined.

To see that $\ASpec{F}$ is injective, let $H$ and $H'$ be monoform objects in $\c{A}$ for which $FH$ and $FH'$ are atom equivalent, i.e. there is a common non-zero subobject $FH \leftarrowtail N \rightarrowtail FH'$. From the fact that $F$ lifts subobjects, and from \rmkref{inj-on-obj}, we get that these monomorphisms are the images under $F$ of monomorphisms $H \leftarrowtail M \rightarrowtail H'$. As $FM = N$ is non-zero, so if $M$. Thus, $H$ and $H'$ are atom equivalent in $\c{A}$.

Next we show that for every object $M \in \c{A}$ there is an equality:
\begin{equation}
  \label{eq:FASupp}
  (\ASpec{F})(\ASupp{M}) \,=\, \ASupp{FM}\;.
\end{equation}

``$\subseteq$'': Let $\a{H} \in \ASupp{M}$, that is, $\a{H} = \a{H'}$ for some monoform object $H'$ of the form $H' \cong L/L'$ where $L' \subseteq L \subseteq M$. We have $(\ASpec{F})(\a{H}) = (\ASpec{F})(\a{H'}) = \a{FH'}$, so we must argue that $\a{FH'}$ is in  $\ASupp{FM}$. As $F$ is exact we have $FL' \subseteq FL \subseteq FM$ and $FH' \cong F(L/L') \cong (FL)/(FL')$ and hence  $\a{FH'} \in \ASupp{FM}$ by definition.

``$\supseteq$'': Let $\a{I} \in \ASupp{FM}$, that is, $\a{I} = \a{I'}$ for some monoform object in $\c{B}$ with $I' \cong N/N'$ where $N' \subseteq N \subseteq FM$. Since $N \subseteq FM$ and $F$ lifts subobjects, there is a subobject $L \subseteq M$ with $FL = N$. Similarly, as $N' \subseteq N = FL$ there is a subobject $L' \subseteq L$ with $FL' = N'$. We now have $L' \subseteq L \subseteq M$ and $F(L/L') \cong (FL)/(FL') \cong N/N' \cong I'$. Since $I'$ is monoform, we conclude from   \eqref{H-FH-monoform} that the object $H:=L/L'$ is monoform, so $\a{H}$ belongs to $\ASupp{M}$. And by constuction, $(\ASpec{F})(\a{H}) = \a{FH} = \a{I'} = \a{I}$.

Recall from \ref{topology} that $\{ \ASupp{M} \ |\, M \in \c{A} \}$ is  an open basis of the topology on $\ASpec{\c{A}}$ (and similarly for $\ASpec{\c{B}}$). It is therefore evident from \eqref{FASupp} that $\ASpec{F}$ is an open map.

Furthermore, to show that $\ASpec{F}$ is continuous, it suffices to show that for any $N \in \c{B}$, the set $\upPhi:=(\ASpec{F})^{-1}(\ASupp{N})$ is open in $\ASpec{\c{A}}$. To see this, let $\a{H} \in \upPhi$ be given. This means that $\a{FH} \in \ASupp{N}$, so \obsref{inclusion} shows that $\a{FH} = \a{K}$ for some monoform object $K$ with $\ASupp{K} \subseteq \ASupp{N}$. By definition of atom equivalence, $FH$ and $K$ have a common non-zero subobject, $K'$, and as $F$ lifts subobjects we have $K' = FH'$ for some non-zero subobject $H'$ of $H$. Clearly $H' \in \a{H}$. Furthermore one has
\begin{displaymath}
  (\ASpec{F})(\ASupp{H'}) \,=\, \ASupp{FH'} \,=\, \ASupp{K'} \,\subseteq\, \ASupp{K} \,\subseteq\, \ASupp N\;,
\end{displaymath}
which means that $\ASupp{H'} \subseteq \upPhi$. Thus $\upPhi$ is open by \ref{topology}.

From the continuity and from \lemref{order-preserving} we get that $\ASpec{F}$ is order-preserving. 
\end{proof}

In the proposition above we considered a full, faithful, and exact functor that lifts subobjects. The result below gives an alternative characterization of such functors. Recall that the \emph{essential image} of a functor $F \colon \c{A} \to \c{B}$ is the smallest full subcategory, $\EssIm{F}$, of $\c{B}$ which contains the image of $F$ and is closed under isomorphisms. We say that $\EssIm{F}$ is \emph{closed under subobjects}, respectively, \emph{closed under quotient objects}, provided that the situation $B \subseteq I \in \EssIm{F}$ in $\c{B}$ implies $B \in \EssIm{F}$, respectively, $I/B \in \EssIm{F}$.

\begin{lem}
  \label{lem:referee-comment}
  Let $F \colon \c{A} \to \c{B}$ be a full, faithful, and additive functor between abelian categories. Such a functor $F$ is exact and lifts subobjects if and only if\,  $\EssIm{F}$ is closed under subobjects. In this case, $\EssIm{F}$ is also closed under quotient objects.
\end{lem}

\begin{proof}
  If $F$ is exact and lifts subobjects, then clearly $\EssIm{F}$ is closed under both subobjects and quotient objects. Conversely, assume that $\EssIm{F}$ is closed under subobjects. Since $F$ is faithful it reflects monomorphisms, that is, if $\alpha$ is a morphism for which $F\alpha$ is mono, then $\alpha$ itself is mono. From this it follows that $F$ lifts subobjects, since $F$ is full. To see that $F$ is exact, let $A' \ira[\alpha']{\longrightarrow} A \ira[\alpha]{\longrightarrow} A''$ be an exact sequence in $\c{A}$. Consider the diagrams,
\begin{displaymath}
  \begin{gathered}
  \xymatrix@!=0.8pc{
    A' \ar@{->>}[dr] \ar[rr]^-{\alpha'} 
    & & 
    A \ar[rr]^-{\alpha} 
    & & 
    A''
    \\
    {} & 
    \Im{\alpha'} \ar@{>->}[ur]^-{\iota} \ar[rr]^-{\varphi}_-{\cong} 
    & & 
    \Ker{\alpha} \ar@{>->}[ul]_-{\kappa} \ar[ur]_-{0} 
    & {}
  }
  \end{gathered}  
  \quad \text{and} \quad
  \begin{gathered}  
  \xymatrix@!=0.8pc{
    FA' \ar@{->>}[dr]_-{\pi} \ar[rr]^-{F\alpha'} 
    & & 
    FA \ar[rr]^-{F\alpha} 
    & & 
    FA''
    \\
    {} & 
    \Im{F\alpha'} \ar@{>->}[ur]^-{\varepsilon} \ar@{>->}[rr]^-{\psi}
    & & 
    \Ker{F\alpha}\;, \ar@{>->}[ul]_-{\lambda} \ar[ur]_-{0} 
    & {}
  }  
\end{gathered}    
\end{displaymath}  
where $\varphi$ is an isomorphism by exactness of the given sequence. As $F$ lifts subobjects there is a mono $\lambda' \colon K \rightarrowtail A$ with $FK=\Ker{F\alpha}$ and $F\lambda' = \lambda$, and further a mono $\psi' \colon I \rightarrowtail K$ with $FI=\Im{F\alpha'}$ and $F\psi' = \psi$. We will show that $\psi'$ has a right inverse; hence it is an isomorphism and therefore so is $\psi=F\psi'$, as desired. As $F$ is full there exists $\pi' \colon A' \to I$ with $F\pi' = \pi$. Note that $\lambda'\psi' \colon I \rightarrowtail A$ is mono and satisfies $\lambda'\psi'\pi' = \alpha'$ since $F(\lambda'\psi'\pi') = \lambda \psi \pi = \varepsilon \pi = F\alpha'$. By the universal property of the image, there exists a unique $\theta \colon \Im{\alpha'} \to I$ with $\lambda'\psi'\theta = \iota$. We also have $F(\alpha\lambda') = F\alpha \circ \lambda = 0$ and hence $\alpha\lambda'=0$, so by the universal property of the kernel there is a unique $\eta \colon K \to \Ker{\alpha}$ with $\kappa\eta = \lambda'$. Now $\theta\mspace{1mu}\varphi^{-1}\eta \colon K \to I$ is a right inverse of $\psi'$, indeed, $\lambda'\psi'\theta\mspace{1mu}\varphi^{-1}\eta = \iota\varphi^{-1}\eta = \kappa\eta = \lambda' = \lambda'\,\mathrm{id}_K$ and $\lambda'$ is mono.
\end{proof}

\begin{ipg}
  \label{disjoint-union}
  Let $\{X_i\}_{i \in I}$ be a family of sets and write $\bigsqcup_{i \in I}X_i$ for the disjoint union. This is the coproduct of $\{X_i\}_{i \in I}$ in the category of sets, so given any family $f_i \colon X_i \to Y$ of maps, there is a unique map $f$ that makes the following diagram commute:
\begin{displaymath}
  \xymatrix{  
    X_i \ar@{>->}[d] \ar[dr]^-{f_i} & {} \\ 
    \bigsqcup_{i \in I}X_i \ar@{..}[r]\ar@{..}[r]\ar@{..}[r]\ar@{..>}[r]_-{f} & Y\;.
  }
\end{displaymath}  

In the case where each $X_i$ is a topological space, $\bigsqcup_{i \in I}X_i$ is equipped with the disjoint union topology, and this yields the coproduct of $\{X_i\}_{i \in I}$ in the category of topological spaces. In fact, for the maps $f_i$ and $f$ in the diagram above, it is well-known that one has:
\begin{displaymath}
  \textnormal{$f$ is continuous (open) \ $\iff$ \ $f_i$ is continuous (open) for every $i \in I$\;.}
\end{displaymath}

If each $X_i$ is a Kolmogorov space, then so is $\bigsqcup_{i \in I}X_i$ (and hence it is the coproduct in the category of Kolmogorov spaces). In this case, and if $Y$ is also a Kolmogorov space, any continuous map $f$ in the diagram above is automatically order-preserving by \lemref{order-preserving}.
\end{ipg}

\begin{prp}
  \label{prp:f}
  Let \mbox{$F_{\mspace{-1mu}i} \colon \c{A}_i \to \c{B}$} $(i \in I)$ be a family of full, faithful, and exact functors between abelian categories that lift subobjects. There exists a unique map $f$ that makes the following diagram commute:
\begin{displaymath}
  \xymatrix{  
    \ASpec{\c{A}_i} \ar@{>->}[d] \ar[dr]^-{\ASpec{F_{\mspace{-1mu}i}}} & {} \\ 
    \bigsqcup_{i \in I}\ASpec{\c{A}_i} \ar@{..}[r]\ar@{..}[r]\ar@{..}[r]\ar@{..>}[r]_-{f} & \ASpec{\c{B}}\;.
  }
\end{displaymath}    
That is, $f$ maps $\a{H} \in \ASpec{\c{A}_i}$ to $\a{F_{\mspace{-1mu}i}\mspace{1mu}H} \in \ASpec{\c{B}}$. This map $f$ is continuous, open, and order-preserving.
\end{prp}

\begin{proof}
  Immediate from \prpref{ASpecF} and \ref{disjoint-union}.
\end{proof}

Our next goal is to find conditions on the functors $F_{\mspace{-1mu}i}$ from \prpref{f} which ensure that the map $f$ is bijective, and hence a homeomorphism and an order-isomorphism.

\begin{thm}
  \label{thm:bijective}
  Let \mbox{$F_{\mspace{-1mu}i} \colon \c{A}_i \to \c{B}$} $(i \in I)$ be a family of functors as in \prpref{f} and consider the induced continuous, open, and order-preserving map 
\begin{displaymath}  
  f \colon \textstyle\bigsqcup_{i \in I}\,\ASpec{\c{A}_i} \longrightarrow \ASpec{\c{B}}\;.
\end{displaymath}  
The map $f$ is injective provided that the following condition holds:
\begin{prt}
\item For $i \neq j$ and $A_i \in \c{A}_i$ and $A_{\mspace{-2mu}j } \in \c{A}_j$ the only common subobject of $F_{\mspace{-1mu}i}\mspace{1mu}A_i$ and \mbox{$F_{\mspace{-4mu}j}\mspace{1mu}A_{\mspace{-2mu}j}$} is $0$.
\end{prt}    
The map $f$ is surjective provided that each $F_{\mspace{-1mu}i}$ has a right adjoint $G_i$ satisfying:
\begin{prt}
\setcounter{prt}{1}
\item For every $B\neq 0$ in $\c{B}$ there exists $i \in I$ with $G_iB \neq 0$.
\end{prt}  
Thus, if \textnormal{(a)} and \textnormal{(b)} hold, then $f$ is a homeomorphism and an order-isomorphism.
\end{thm}

\begin{proof}
  First we show that condition (a) implies injectivity of $f$. Let $\a{H} \in \ASpec{\c{A}_i}$~and $\a{H'} \in \ASpec{\c{A}_j}$ be arbitrary elements in $\bigsqcup_{i \in I}\ASpec{\c{A}_i}$ with $f(\a{H}) = f(\a{H'})$, that is, $\a{F_{\mspace{-1mu}i}\mspace{1mu}H} = \a{F_{\mspace{-4mu}j}\mspace{1mu}H'}$. This means that the monoform objects $F_{\mspace{-1mu}i}\mspace{1mu}H$ and $F_{\mspace{-4mu}j}\mspace{1mu}H'$ are atom equivalent, so they contain a common non-zero subobject $N$. By the assumption (a), we must have $i=j$.  As the map $\ASpec{F_{\mspace{-1mu}i}}$ is injective, see \prpref{ASpecF}, we conclude that $\a{H} = \a{H'}$.
  
Next we show that condition (b) implies surjectivity of $f$. First note that for every $i \in I$ and $B \in \c{B}$ the counit $\varepsilon_B \colon F_{\mspace{-1mu}i}\mspace{1mu}G_iB \to B$ is a monomorphism. Indeed, consider the subobject $\iota \colon \Ker{\varepsilon_B} \rightarrowtail F_{\mspace{-1mu}i}\mspace{1mu}G_iB$. Since $F_{\mspace{-1mu}i}$ lifts subobjects there is a monomorphism $\iota' \colon K \rightarrowtail G_iB$ with $F_{\mspace{-1mu}i}K = \Ker{\varepsilon_B}$ and $F_{\mspace{-1mu}i}(\iota') = \iota$. Applying the functor $\Hom{\c{B}}{F_{\mspace{-1mu}i}K}{-}$ to the exact se\-quen\-ce $0 \to F_{\mspace{-1mu}i}K \to F_{\mspace{-1mu}i}\mspace{1mu}G_iB \to B$, we get an exact sequence,
\begin{equation}
  \label{eq:varepsilon}
  \xymatrix@C=1.7pc{
    0 \ar[r] & \Hom{\c{B}}{F_{\mspace{-1mu}i}K}{F_{\mspace{-1mu}i}K} \ar[r] & \Hom{\c{B}}{F_{\mspace{-1mu}i}K}{F_{\mspace{-1mu}i}\mspace{1mu}G_iB} \ar[r]^-{\varepsilon_B \,\circ\, -} & \Hom{\c{B}}{F_{\mspace{-1mu}i}K}{B}
  }.
\end{equation}
The rightmost map is an isomorphism, indeed, \cite[IV\S1]{Mac} yields a commutative diagram
\begin{displaymath}
  \xymatrix{
    \Hom{\c{A}}{K}{G_iB} \ar[d]_-{F_{\mspace{-1mu}i}(-)}^-{\cong} \ar[dr]^-{\ \ \ \ \ \varphi^{-1}_{K,B} \textnormal{ \ (adjunction)}}_-{\cong}
    \\
    \Hom{\c{B}}{F_{\mspace{-1mu}i}K}{F_{\mspace{-1mu}i}\mspace{1mu}G_iB} \ar[r]_-{\varepsilon_B \,\circ\, -} & \Hom{\c{B}}{F_{\mspace{-1mu}i}K}{B}\;.
  }
\end{displaymath}
In this diagram, the vertical map is an isomorphism as $F_{\mspace{-1mu}i}$ is assumed to be full and faithful, so it follows that $\varepsilon_B \circ -$ is an isomorphism. Now \eqref{varepsilon} shows that $\Hom{\c{B}}{F_{\mspace{-1mu}i}K}{F_{\mspace{-1mu}i}K}=0$ and hence $F_{\mspace{-1mu}i}K = \Ker{\varepsilon_B} = 0$.

To see that $f$ is surjective, let $H$ be any monoform object in $\c{B}$. As $H \neq 0$ there exists by (b) some $i \in I$ with $G_iH \neq 0$. This implies  $F_{\mspace{-1mu}i}\mspace{1mu}G_iH \neq 0$, see \rmkref{inj-on-obj}. As just proved, $F_{\mspace{-1mu}i}\mspace{1mu}G_iH$ is a subobject of $H$, so in this case
$F_{\mspace{-1mu}i}\mspace{1mu}G_iH$ is a non-zero subobject of the monoform object $H$. Thus \cite[Prop.~2.2]{Kanda-Serre} implies that $F_{\mspace{-1mu}i}\mspace{1mu}G_iH$ is a monoform object, atom equivalent to $H$. From \eqref{H-FH-monoform} we get that $G_iH$ is monoform, so $\a{G_iH}$ is an element in $\ASpec{\c{A}_i}$ satisfying $f(\a{G_iH}) = \a{F_{\mspace{-1mu}i}\mspace{1mu}G_iH} = \a{H}$.
\end{proof}

\section{Application to quiver representations}
\label{sec:Application-quiver}

In this section, we will apply \thmref{bijective} to compute the atom spectrum of the category of $\c{A}$-valued representations of any (well-behaved) quiver with relations $(Q,\c{R})$. Here $\c{A}$ is a $\Bbbk$-linear abelian category and $\Bbbk$ is any commutative ring. \appref{Quiver} contains some background material on quivers with relations and their representations needed in this section. The main result is \thmref{quiver}, and we also prove Theorem~A from the Introduction.
  
Enochs, Estrada, and Garc{\'{\i}}a Rozas define in \cite[Sect.~4]{EEGR09} what it means for a quiver,~\textsl{without} relations, to be right rooted. Below we extend their definition to quivers \textsl{with} relations. To parse the following, recall the notion of the $\Bbbk$-linearization of a category and that of an ideal in a $\Bbbk$-linear category, as described in \ref{app-k-linear} and \ref{app-relations}.
    
\begin{dfn}
  \label{dfn:right-rooted}
  A quiver with relations $(Q,\c{R})$ is said to be \emph{right rooted} if for every infinite sequence of (not necessarily different) composable arrows in $Q$,
\begin{displaymath}
  \xymatrix@C=1.3pc{
    \bullet \ar[r]^-{a_1} & \bullet \ar[r]^-{a_2} & \bullet \ar[r]^-{a_3} & \cdots
    },
\end{displaymath}    
there exists $N \in \mathbb{N}$ such that the path $a_N\cdots a_1$ (which is a morphism in the category $\Bbbk \bar{Q}$) belongs to the two-sided ideal $(\c{R}) \subseteq \Bbbk \bar{Q}$.
\end{dfn}
  
\begin{obs}
  \label{obs:extend}
  Let $(Q,\c{R})$ be a quiver with relations. If there exists \textsl{no} infinite sequence $\bullet \xrightarrow{\ \ \ } \bullet \xrightarrow{\ \ \ } \bullet \xrightarrow{\ \ \ } \cdots$ of (not necessarily different) composable arrows in $Q$, then $(Q,\c{R})$ is right rooted, as the requirement in \dfnref{right-rooted} becomes void. If $Q$ is a quiver without relations, i.e.~$\c{R}=\emptyset$ and hence $(\c{R})=\{0\}$, then $Q$ is right rooted if and only if there exists no such infinite sequence $\bullet \xrightarrow{\ \ \ } \bullet \xrightarrow{\ \ \ } \bullet \xrightarrow{\ \ \ } \cdots$; indeed, a path $a_N\cdots a_1$ is never zero in the absence of relations. Consequently, our \dfnref{right-rooted} of right rootedness for quivers with relations extends the similar definition for quivers without relations found in \cite[Sect.~4]{EEGR09}.
\end{obs}  
              
Next we introduce admissible relations and stalk functors.

\begin{dfn}   
  \label{dfn:admissible} 
  A relation $\rho$ in a quiver $Q$, see \ref{app-relations}, is called \emph{admissible} if the coefficient in the linear combination $\rho$ to every trivial path $e_i$ ($i \in Q_0$) is zero. We refer to a quiver with relations $(Q,\c{R})$ as a \emph{quiver with admissible relations} if every relation in $\c{R}$ is admissible.
\end{dfn}

As we shall be interested in right rooted quivers with admissible relations, it seems in order to compare these notions with the more classic notion of admissibility:

\begin{rmk}
According to \cite[Chap.~II.2 Def.~2.1]{ASS1}, a set $\c{R}$ of relations in a quiver $Q$ with finitely many vertices is \emph{admissible} if $\mathfrak{a}^m \subseteq (\c{R}) \subseteq \mathfrak{a}^2$ holds for some $m \geqslant 2$. Here $\mathfrak{a}$ is the \emph{arrow ideal} in $\Bbbk Q$, that is, the two-sided ideal generated by all arrows in $Q$. Note that:
\begin{displaymath}
  \begin{array}{c}
   \textnormal{$\c{R}$ is admissible as in}
   \\
   \textnormal{\cite[Chap.~II.2 Def.~2.1]{ASS1}}   
  \end{array}
  \ \implies \
  \begin{array}{c}
   \textnormal{$\c{R}$ is admissible as in \dfnref{admissible} and}
   \\
   \textnormal{$(Q,\c{R})$ is right rooted as in \dfnref{right-rooted}.}   
  \end{array}  
\end{displaymath}  
Indeed, in terms of the arrow ideal, our definition of admissibility simply means $(\c{R}) \subseteq \mathfrak{a}$\footnote{\,Often, not much interesting comes from considering relations in $\mathfrak{a} \smallsetminus \mathfrak{a}^2$. To illustrate this point, consider e.g.~the Kronecker quiver \smash{$K_2 = \xymatrix@C=1.2pc{\bullet \ar@<0.7ex>[r]|-{a} \ar@<-0.7ex>[r]|-{b} & \bullet}$}\! with one relation $\rho:=a-b \in \mathfrak{a} = (a,b) \subseteq \Bbbk K_2$. Clearly, the category $\Rep{(K_2,\{\rho\})}{\c{A}}$ is equivalent to $\Rep{A_2}{\c{A}}$ where \smash{$A_2 = \xymatrix@C=1pc{\bullet \ar[r] & \bullet}$}. So the representation theory of $(K_2,\{\rho\})$ is already covered by the  representation theory of a quiver (in this case, $A_2$) with relations (in this case, $\c{R}=\emptyset$) contained in the square of the arrow ideal.}, and if there is an inclusion $\mathfrak{a}^m \subseteq (\c{R})$, then \dfnref{right-rooted} holds with (universal) $N=m$.

If $(Q,\c{R})$ is right rooted, one does not necessarily have $\mathfrak{a}^m \subseteq (\c{R})$ for some $m$. Indeed, let $Q$ be a quiver with one vertex and countably many loops $x_1,x_2,\ldots$. For each $\ell \geqslant 1$ let $\c{R}_\ell = \{x_{u_\ell}\cdots x_{u_1}x_\ell \,|\, u_1,\ldots,u_\ell \in \mathbb{N} \}$ be the set of all paths of length $\ell+1$ starting with $x_\ell$. Set $\c{R} = \bigcup_{\ell \geqslant 1} \c{R}_\ell$. Evidently, \smash{$(\c{R}) \subseteq \mathfrak{a}^2 = (x_1,x_2,\ldots)^2$} and $(Q,\c{R})$ is right rooted. As none of the elements $x_1, x_2^2, x_3^3,\ldots$ belong to $(\c{R})$ we have $\mathfrak{a}^m \nsubseteq (\c{R})$ for every $m$.

However, if $Q$ has only finitely many arrows (in addition to having only finitely many vertices), then right rootedness of $(Q,\c{R})$ means precisely that $\mathfrak{a}^m \subseteq (\c{R})$ for some $m$.
\end{rmk}

\begin{dfn}  
  \label{dfn:Si}
   Let $Q$ be a quiver and let $\c{A}$ be an abelian category. For every $i \in Q_0$ there is a \emph{stalk functor} $S_{\mspace{-3mu}i} \colon \c{A} \to \Rep{Q}{\c{A}}$ which assigns to $A \in \c{A}$ the \emph{stalk representation} $S_{\mspace{-3mu}i}\mspace{1mu}A$ given by $(S_{\mspace{-3mu}i}\mspace{1mu}A)(j) = 0$ for every vertex $j \neq i$ in $Q_0$ and $(S_{\mspace{-3mu}i}\mspace{1mu}A)(i) = A$. For every path $p \neq e_i$ in $Q$ one has $(S_{\mspace{-3mu}i}\mspace{1mu}A)(p)=0$  and, of course, $(S_{\mspace{-3mu}i}\mspace{1mu}A)(e_i)=\mathrm{id}_A$.
\end{dfn}   

\begin{rmk}
  \label{rmk:admissible-stalk}
Let $\rho$ be a relation in a quiver $Q$ and let $x_i \in \Bbbk$ be the coefficient (which may or may not be zero) to the path $e_i$ in the linear combination $\rho$. If $A$ is any object in a $\Bbbk$-linear abelian category $\c{A}$, then  $(S_{\mspace{-3mu}i}\mspace{1mu}A)(\rho)=x_i\mspace{2mu}\mathrm{id}_A$. It follows that the stalk representation $S_{\mspace{-3mu}i}\mspace{1mu}A$ satisfies every admissible relation. Thus, if $(Q,\c{R})$ be a quiver with admissible relations, then every $S_{\mspace{-3mu}i}$ can be viewed as a functor $\c{A} \to \Rep{(Q,\c{R})}{\c{A}}$.
\end{rmk}
                                                               
\begin{ipg}  
  \label{Ki}
   Let $(Q,\c{R})$ be a quiver with admissible relations and let $\c{A}$ be a $\Bbbk$-linear abelian category. For every $i \in Q_0$ the stalk functor $S_{\mspace{-3mu}i} \colon \c{A} \to \Rep{(Q,\c{R})}{\c{A}}$ from \rmkref{admissible-stalk} has a right adjoint, namely the functor $K_i \colon \Rep{(Q,\c{R})}{\c{A}} \to \c{A}$ given by
\begin{displaymath}
  K_i\mspace{1mu}X\,=\!\bigcap_{a \colon \mspace{-2mu}i \to j}\!\!\Ker{X(a)}
  \,=\, \Ker{\Big(X(i) \stackrel{\psi^X_i}{\longrightarrow} \!\!\! \prod_{a \colon \mspace{-2mu}i \to j}\!\!\! X(j)\Big)}
  \;,
\end{displaymath}  
where the intersection/product is taken over all arrows $a \colon \mspace{-2mu}i \to j$ in $Q$ with source $i$, and \smash{$\psi^X_i$} is the morphism whose $a^\mathrm{th}$ coordinate function is $X(a) \colon X(i) \to X(j)$. For a quiver without relations ($\c{R}=\emptyset$), the adjunctions $(S_{\mspace{-3mu}i},K_i)$ were established in \cite[Thm.~4.5]{HJ-quiver}, but evidently this also works for quivers with admissible relations.

Note that the existence of $K_i$ requires that the product $\prod_{a \colon \mspace{-2mu}i \to j}$ can be formed in $\c{A}$; this is the case if, for example, $\c{A}$ is complete (satisfies AB3*) or if $\c{A}$ is arbitrary but there are only finitely many arrows in $Q$ with source $i$. We tacitly assume that each $K_i$ exists. 
\end{ipg}

\begin{lem}
  \label{lem:b}
  Let $(Q,\c{R})$ be a quiver with admissible relations, let $\c{A}$ be a $\Bbbk$-linear abelian category, and let $X \in \Rep{(Q,\c{R})}{\c{A}}$. If $X \neq 0$ and $K_iX=0$ holds for all $i \in Q_0$, then there exists an infinite sequence of (not necessarily different) composable arrows in $Q$,
\begin{equation}
  \label{eq:an}
  \begin{gathered}
  \xymatrix@C=1.3pc{
    \bullet \ar[r]^-{a_1} & \bullet \ar[r]^-{a_2} & \bullet \ar[r]^-{a_3} & \cdots
    },
  \end{gathered}    
\end{equation}    
such that $X(a_n\cdots a_1) \neq 0$ for every $n \geqslant 1$. In particular, if $(Q,\c{R})$ is right rooted and $X \neq 0$, then $K_iX \neq 0$ for some $i \in Q_0$.
\end{lem}

\begin{proof}
  As $X \neq 0$ we have $X(i_1) \neq 0$ for some vertex $i_1$. As $K_{i_1}\mspace{-1mu}X=0$ we have $X(i_1) \nsubseteq K_{i_1}\mspace{-1mu}X$ so there is at least one arrow $a_1 \colon i_1 \to i_2$ with $X(i_1) \nsubseteq \Ker{X(a_1)}$, and hence $X(a_1) \neq 0$. As $0 \neq \Im{X(a_1)} \subseteq X(i_2)$ and $K_{i_2}\mspace{-1mu}X=0$ we have $\Im{X(a_1)} \nsubseteq K_{i_2}\mspace{-1mu}X$, so there is at least one arrow $a_2 \colon i_2 \to i_3$ such that $\Im{X(a_1)} \nsubseteq \Ker{X(a_2)}$. This means that $X(a_2) \circ X(a_1) = X(a_2a_1)$ is non-zero. Continuing in this manner, the first assertion in the lemma follows.
  
For the second assertion, assume that there is some $X \neq 0$ with $K_iX=0$ for all $i \in Q_0$.  By the first assertion there exists an infinite sequence of composable arrows \eqref{an} such that $X(a_n\cdots a_1) \neq 0$ for every $n \geqslant 1$. Hence $a_n\cdots a_1 \notin (\c{R}) \subseteq \Bbbk \bar{Q}$ holds for every $n \geqslant 1$ by the lower equivalence in the diagram in \ref{app-relations} Thus
$(Q,\c{R})$ is not right rooted.
\end{proof}

The result below shows that for a \textsl{right rooted} quiver with \textsl{admissible} relations $(Q,\c{R})$, the atom spectrum of $\Rep{(Q,\c{R})}{\c{A}}$ depends only on the atom spectrum of $\c{A}$ and on the (cardinal) number of vertices in $Q$. The arrows and the relations in $Q$ play no (further)~role! 

\enlargethispage{6ex}

\begin{thm}
  \label{thm:quiver}
    Let $(Q,\c{R})$ be a quiver with admissible relations and let $\c{A}$ be any $\Bbbk$-linear abelian category. There is an injective, continuous, open, and order-preserving map,
\begin{displaymath}  
  f \colon \textstyle\bigsqcup_{i \in Q_0} \ASpec{\c{A}} \longrightarrow \ASpecp{\Rep{(Q,\c{R})}{\c{A}}}\;,
\end{displaymath}  
given by $\textnormal{($i^\mathrm{th}$ copy of $\ASpec{\c{A}}$) $\ni \a{H}$} 
  \,\longmapsto\, \a{S_{\mspace{-3mu}i}\mspace{1mu}H}$. If, in addition, $(Q,\c{R})$ is right rooted, then $f$ is also surjective, and hence it is a homeomorphism and an order-isomorphism.
\end{thm}

\begin{proof}
 We apply \thmref{bijective} to the functors $F_{\mspace{-1mu}i}=S_{\mspace{-3mu}i}$ and $G_i=K_i$ ($i \in Q_0$) from \dfnref[]{Si} and \ref{Ki}. The functor $S_{\mspace{-3mu}i}$ is obviously exact, and it also lifts subobjects as every subobject of $S_{\mspace{-3mu}i}\mspace{1mu}A$ has the form $S_{\mspace{-3mu}i}\mspace{1mu}A'$ for a subobject $A' \rightarrowtail A$ in $\c{A}$. It is immediate from the definitions that the~unit $\mathrm{id}_{\c{A}} \to K_i\mspace{1mu}S_{\mspace{-3mu}i}$ of the adjunction $(S_{\mspace{-3mu}i},K_i)$ is an isomorphism, and hence $S_{\mspace{-3mu}i}$ is full and faithful by (the dual of) \cite[IV.3, Thm.~1]{Mac}. Hence the functors $S_{\mspace{-3mu}i}$ meet the requirements in \prpref{f} and we get that $f$ is well-defined, continuous, open, and order-preserving. 
 
 Evidently condition (a) in \thmref{bijective} holds, so $f$ is injective. If $(Q,\c{R})$ is right rooted, then \lemref{b} shows that condition (b) in \thmref{bijective} holds, so $f$ is surjective.
\end{proof}
      
\begin{exa}
  \label{exa:Ch}
  The quiver (without relations):
  \begin{displaymath}
    A^\infty_\infty \,\colon \ \     
    \begin{gathered}        
    \xymatrix@!=0.5pc{
      \cdots \ar[r] & 
      \underset{2}{\bullet} \ar[r]^-{d_2} &       
      \underset{1}{\bullet} \ar[r]^-{d_1} & 
      \underset{0}{\bullet} \ar[r]^-{d_0} &
      \underset{-1}{\bullet} \ar[r]^-{d_{-1}} &      
      \underset{-2}{\bullet} \ar[r]^-{d_{-2}} &            
      \cdots
    }
    \end{gathered}        
  \end{displaymath}  
  is not right rooted, but when equipped with the admissible relations $\c{R} = \{d_{n-1}d_n \,|\, n \in \mathbb{Z}\}$ it becomes right rooted. For any ($\mathbb{Z}$-linear) abelian category $\c{A}$, the category $\Rep{(A^\infty_\infty,\c{R})}{\c{A}}$ is equivalent to the category $\Ch{\c{A}}$ of chain complexes in $\c{A}$. Hence \thmref{quiver} yields a homeomorphism and an order-isomorphism
\begin{displaymath}  
  \textstyle\bigsqcup_{i \in \mathbb{Z}}\, \ASpec{\c{A}} \longrightarrow \ASpecp{\Ch{\c{A}}}
\end{displaymath}  
given by $\textnormal{($i^\mathrm{th}$ copy of $\ASpec{\c{A}}$) $\ni \a{H}$}
  \,\longmapsto\, \a{\cdots \to 0 \to 0 \to H \to 0 \to 0 \to \cdots}$ with $H$ in degree $i$ and zero in all other degrees.  
\end{exa}

The next goal is to apply \thmref{quiver} to prove Theorem~A from the Introduction.
      
\begin{dfn}
  \label{dfn:p-tilde}
Let $Q$ be a quiver with finitely many vertices. For every ideal $\mathfrak{p}$ in $\Bbbk$ and every vertex $i$ in $Q$ set $\tilde{\mathfrak{p}}(i) = \{\mspace{1mu} \xi \in \Bbbk Q \mspace{3mu}|\mspace{3mu} \textnormal{the coefficient to $e_i$ in $\xi$ belongs to $\mathfrak{p}$} \mspace{1mu}\}$.
\end{dfn}

\begin{lem}
  \label{lem:pi-ideal}
  With the notation above, the set $\tilde{\mathfrak{p}}(i)$ is a two-sided ideal in $\Bbbk Q$ which contains every admissible relation.
\end{lem}

\begin{proof}
  Let $p \neq e_i$ be a path in $Q$ and let $\xi$ be an element in $\Bbbk Q$. In the linear combinations $p\xi$ and $\xi p$ the coefficient to $e_i$ is zero. In the linear combinations $e_i\xi$ and $\xi e_i$ the coefficient to $e_i$ is the same as the  coefficient to $e_i$ in the given element $\xi$. Hence $\tilde{\mathfrak{p}}(i)$ is a two-sided ideal in $\Bbbk Q$. By \dfnref{admissible}, every admissible relation belongs to $\tilde{\mathfrak{p}}(i)$.
\end{proof}

\begin{proof}[Proof of Theorem~A]
Let $\tilde{f}$ be the map defined by commutativity of the diagram:
\begin{equation}
  \label{eq:tilde-f}
  \begin{gathered}
  \xymatrix{
    \bigsqcup_{i \in Q_0}\Spec{\Bbbk} \ar@{..}[r]\ar@{..}[r]\ar@{..}[r]\ar@{..>}[r]^-{\tilde{f}}
    \ar[d]_-{\bigsqcup_{i \in Q_0}q}^-{\sim} & 
    \ASpecp{\Mod{\Bbbk Q/I}}
    \\
    \bigsqcup_{i \in Q_0}\ASpecp{\Mod{\Bbbk}} \ar[r]^-{f} & 
    \ASpecp{\Rep{(Q,\c{R})}{\Mod{\Bbbk}}}\;.\ar[u]^-{\sim}_-{\ASpec{U}}
  }
  \end{gathered}  
\end{equation} 
Here the lower horizontal map is the map from \thmref{quiver} with $\c{A}=\Mod{\Bbbk}$; the left vertical map is the order-isomorphism and homeomorphism described in \ref{Spec}; and the right vertical map is induced by the equivalence of categories $U \colon \Rep{(Q,\c{R})}{\Mod{\Bbbk}} \to \Mod{\Bbbk Q/I}$ given in \ref{app-path-algebra}. An element $\mathfrak{p} \in \textnormal{($i^\mathrm{\,th}$ copy of $\Spec{\Bbbk}$)}$ is by \smash{$\bigsqcup_{i \in Q_0}q$} mapped to the atom $\a{\Bbbk/\mathfrak{p}} \in \textnormal{($i^\mathrm{\,th}$ copy of $\ASpecp{\Mod{\Bbbk}}$)}$, which by $f$ is mapped to the atom $\a{S_{\mspace{-3mu}i}(\Bbbk/\mathfrak{p})}$. The functor $U$ maps the representation $S_{\mspace{-3mu}i}(\Bbbk/\mathfrak{p})$ to the left $\Bbbk Q/I$-module (= a left $\Bbbk Q$-module killed by $I$) whose underlying $\Bbbk$-module is $\Bbbk/\mathfrak{p}$ (more precisely, $0 \oplus \cdots \oplus 0 \oplus \Bbbk/\mathfrak{p} \oplus 0 \oplus \cdots \oplus 0$ with a ``$0$'' for each vertex $\neq i$) on which $e_i$ acts as the identity and $p \cdot \Bbbk/\mathfrak{p} = 0$ for all paths $p \neq e_i$. This means that the left $\Bbbk Q/I$-module $U(S_{\mspace{-3mu}i}(\Bbbk/\mathfrak{p}))$ is isomorphic to $\Bbbk Q / \tilde{\mathfrak{p}}(i)$. Indeed, $\Bbbk Q / \tilde{\mathfrak{p}}(i)$ \textsl{is} a $\Bbbk Q/I$-module as $I \subseteq \tilde{\mathfrak{p}}(i)$ by \lemref{pi-ideal}; and as a $\Bbbk$-module it is isomorphic to $\Bbbk/\mathfrak{p}$ since the $\Bbbk$-linear map $\Bbbk Q \to \Bbbk/\mathfrak{p}$ given by $\xi \mapsto [\textnormal{(coefficient to $e_i$ in $\xi$)}]$ has kernel $\tilde{\mathfrak{p}}(i)$. As noted in the proof of \lemref{pi-ideal}, every path $p \neq e_i$ multiplies $\Bbbk Q$ into $\tilde{\mathfrak{p}}(i)$, so one has $p \cdot \Bbbk Q / \tilde{\mathfrak{p}}(i) = 0$, and $e_i$ acts as the identity on $\Bbbk Q / \tilde{\mathfrak{p}}(i)$. Having proved the isomorphism $U(S_{\mspace{-3mu}i}(\Bbbk/\mathfrak{p})) \cong \Bbbk Q / \tilde{\mathfrak{p}}(i)$, it follows that $\tilde{f}(\mathfrak{p}) = \a{\Bbbk Q / \tilde{\mathfrak{p}}(i)}$. Thus $\tilde{f}$ acts as described in the theorem. The assertions about $\tilde{f}$ follow from the commutative diagram \eqref{tilde-f} and from the properties of the map $f$ given in \thmref{quiver}.
\end{proof}

Below we examine the map $\tilde{f}$ from Theorem~A in some concrete examples. 
        
\begin{exa}
  \label{exa:subspace}
Consider the $(n-\!1)$-subspace quiver (no relations), which is right rooted:
  \begin{displaymath}
    \upSigma_n : \qquad
    \begin{gathered}
    \xymatrix@R=0.7pc@C=1.5pc{
    {} & {} & \overset{n}{\bullet} & {}
    \\
    \underset{1}{\bullet} \ar[urr]^-{a_1} & \underset{2}{\bullet} \ar[ur]_-{a_2} & \cdots & \underset{n-1}{\bullet} \ar[ul]_-{a_{n-1}}
    }
    \end{gathered}    
  \end{displaymath}  
  The path algebra $\Bbbk\upSigma_n$ is isomorphic to the following $\Bbbk$-subalgebra of $\mathrm{M}_n(\Bbbk)$:
\begin{displaymath}
  \mathrm{L}_n(\Bbbk) \,=\, \{\, (x_{ij}) \in \mathrm{M}_n(\Bbbk) \ | \ x_{ij} = 0 \textnormal{ if $i \neq n$ and $i \neq j$} \ \}\;.
\end{displaymath}  
 Under this isomorphism the arrow $a_i$ in $\upSigma_n$ corresponds to the matrix $\alpha_i \in \mathrm{L}_n(\Bbbk)$ with $1$ in entry $(n,i)$ and $0$ elsewhere, and the trivial path $e_i$ corresponds to the matrix $\varepsilon_i \in \mathrm{L}_n(\Bbbk)$ with $1$ in entry $(i,i)$ and $0$ elsewhere. It follows that, via this isomorphism, the ideal $\tilde{\mathfrak{p}}(i) \subseteq \Bbbk\upSigma_n$ from \dfnref{p-tilde} is identified with the ideal
\begin{displaymath}
  \bar{\mathfrak{p}}(i) = 
  {\renewcommand{\arraystretch}{0.0} 
  \setlength{\arraycolsep}{2.5pt} 
  \left(
    \begin{array}{ccccc}
       \Bbbk & {} & {} & {} & {} \\
       {} & \ddots & {} & \,\,\text{\LARGE \raisebox{1ex}{$0$}} & {} \\
       \text{\LARGE $0$} & {} & \mathfrak{p} & {} & {} \\ 
       {} & {} & {} & \text{\raisebox{1ex}{$\ddots$}} & {} \\
       \Bbbk & \cdots & \Bbbk & \cdots & \Bbbk 
  \end{array}\right)
  } \subseteq\, \mathrm{L}_n(\Bbbk)\;.
  \vspace*{0.5ex}
\end{displaymath}
Now Theorem~A yields a homeomorphism and an order-isomorphism,
  \begin{displaymath}
    \textstyle\bigsqcup_{i=1}^n\Spec{\Bbbk} \longrightarrow \ASpecp{\Mod{\mathrm{L}_n(\Bbbk)}}
  \end{displaymath} 
given by $\textnormal{($i^\mathrm{th}$ copy of $\Spec{\Bbbk}$) $\ni \mathfrak{p}$} 
  \,\longmapsto\, \a{\mspace{2mu}\mathrm{L}_n(\Bbbk)/\mspace{2mu}\bar{\mathfrak{p}}(i)\mspace{1mu}}$.
\end{exa}

\begin{exa}
  \label{exa:free-algebra}
  Let $Q$ be any quiver with finitely many vertices. Let $m>0$ be any natural number and let $\c{P}_m$ be the relations consisting of all paths in $Q$ of length $m$. Clearly these relations are admissible and $(Q,\c{P}_m)$ is right rooted. If $\mathfrak{a}$ denotes the arrow ideal in $\Bbbk Q$, then $(\c{P}_m) = \mathfrak{a}^m$, so Theorem~A yields a homeomorphism and an order-isomorphism,
\begin{displaymath}  
  \textstyle\bigsqcup_{i \in Q_0} \Spec{\Bbbk} \longrightarrow \ASpecp{\Mod{\Bbbk Q/\mathfrak{a}^m}}\;,
\end{displaymath} 
given by  $\textnormal{($i^\mathrm{th}$ copy of $\Spec{\Bbbk}$) $\ni \mathfrak{p}$} 
  \,\longmapsto\, \a{\mspace{2mu}\Bbbk Q/\mspace{2mu}\tilde{\mathfrak{p}}(i)\mspace{1mu}}$. In the special case where $Q$ is the quiver with one vertex and $n$ loops $x_1,\ldots,x_n$ one has $\Bbbk Q = \Bbbk\langle x_1,\ldots,x_n \rangle$, the free $\Bbbk$-algebra. More\-over, for $\mathfrak{p} \in \Spec{\Bbbk}$ we have $\tilde{\mathfrak{p}} = \mathfrak{p} + (x_1,\ldots,x_n)$ and hence $\Bbbk Q/\mspace{2mu}\tilde{\mathfrak{p}} \cong \Bbbk/\mathfrak{p}$, which is a module over $\Bbbk\langle x_1,\ldots,x_n \rangle$ where all the variables $x_1,\ldots,x_n$ act as zero. Thus the map
\begin{displaymath}  
  \Spec{\Bbbk} \longrightarrow \ASpecp{\Mod{\Bbbk\langle x_1,\ldots,x_n \rangle/(x_1,\ldots,x_n)^m}}
\end{displaymath} 
given by $\Spec{\Bbbk} \ni \mathfrak{p} \,\longmapsto\, \a{\Bbbk/\mathfrak{p}}$ is a homeomorphism and an order-isomorphism.    
\end{exa}

We end with an example illustrating the necessity of the assumptions in Theorem~A. We shall see that the map $\tilde{f}$ need not be surjective if $(Q,\c{R})$ is not right rooted and that the  situation is more subtle when the relations are not admissible.

\begin{exa}
  \label{exa:poly}
  Consider the Jordan quiver (which is not right rooted): \vspace*{0.6ex}
  \begin{displaymath}
    J : \quad
    \xymatrix@C=-1.2ex{
    \bullet & \ar@(ur,dr)[]^X
    } \vspace*{0.6ex}
  \end{displaymath}
  The path algebra $\Bbbk J$ is isomorphic to the polynomial ring $\Bbbk[X]$, which is commutative, so via the homeomorphism and order-isomorphism $q \colon \Spec{\Bbbk[X]} \to \ASpecp{\Mod{\Bbbk[X]}}$ in \ref{Spec}, the map $\tilde{f} \colon \Spec{\Bbbk} \to \ASpecp{\Mod{\Bbbk[X]}}$ from Theorem~A may be identified with a map
  \begin{displaymath}
    \Spec{\Bbbk} \longrightarrow \Spec{\Bbbk[X]}\;.
  \end{displaymath} 
  It is not hard to see that this map is given by \mbox{$\mathfrak{p} \mapsto \mathfrak{p}+(X) = \{f(X) \in \Bbbk[X]\,|\, f(0) \in \mathfrak{p}\}$}, so it is injective but not surjective. Typical prime ideals in $\Bbbk[X]$ that are not of the form $\mathfrak{p}+(X)$ are $\mathfrak{q}[X]$ where $\mathfrak{q} \in \Spec{\Bbbk}$.  
  Also notice that for the Jordan quiver, the functor from \ref{Ki},
\begin{displaymath}
  \xymatrix{
    \Mod{\Bbbk[X]} \,\simeq\, \Rep{J\,}{\Mod{\Bbbk}} \ar[r]^-{K} & \Mod{\Bbbk}
  },
\end{displaymath}  
 maps a $\Bbbk[X]$-module $M$ to \smash{$KM = \Ker{\big(M \stackrel{X\cdot}{\longrightarrow} M\big)}$}. Thus it may happen that $KM = 0$ (if multiplication by $X$ on $M$ is injective) even though $M \neq 0$. This also shows that the last assertion in \lemref{b} can fail for quivers that are not right rooted.
 
Now let $\Bbbk=\mathbb{Z}$ and consider e.g.~the relations $\c{R}=\{X^3,2\}$ in $J$ (where ``$2$'' means two times the trivial path on the unique vertex in $J$). Then $(J,\c{R})$ is right rooted because of the relation $X^3$, however, the relation $2$ is not admissible. In this case,
\begin{displaymath}
    \Rep{(J,\c{R})}{\Mod{\mathbb{Z}}} \,\cong\, \Mod{\mathbb{Z}[X]/(X^3,2)} \,=\, \Mod{\mathbb{F}_2[X]/(X^3)}\;,
\end{displaymath}  
so $\ASpecp{\Rep{(J,\c{R})}{\Mod{\mathbb{Z}}}}$ consists of a single element. This set is not even equipotent to $\Spec{\mathbb{Z}}$, in particular, there exists no homeomorphism or order-isomorphism between $\ASpecp{\Rep{(J,\c{R})}{\Mod{\mathbb{Z}}}}$ and $\Spec{\mathbb{Z}}$.
\end{exa}

\section{Application to comma categories}
\label{sec:Comma}

In this section, we consider the \emph{comma category} $\com{U}{V}$, see \cite[II.6]{Mac}, associated to a pair of additive functors between abelian categories:
\begin{displaymath}
  \xymatrix{
    \c{A} \ar[r]^{U} & \c{C} & \c{B} \ar[l]_-{V}
  }.
\end{displaymath}
An object in $\com{U}{V}$ is a triple $(A,B,\theta)$ where $A \in \c{A}$, $B \in \c{B}$ are objects and $\theta \colon UA \to VB$ is a morphism in $\c{C}$. A morphism $(A,B,\theta) \to (A',B',\theta')$ in $\com{U}{V}$ is a pair of morphisms $(\alpha,\beta)$, where $\alpha \colon A \to A'$ is a morphism in $\c{A}$ and $\beta \colon B \to B'$ is a morphism in $\c{B}$, such that the following diagram commutes:
\begin{displaymath}
  \xymatrix{
    UA \ar[r]^-{U\alpha} \ar[d]_{\theta} & UA'\;\phantom{.} \ar@<-2pt>[d]^-{\theta'}
    \\
    VB \ar[r]^{V\beta} & VB'\;.    
  }
\end{displaymath}
The comma category arising from the special case \!\smash{$\xymatrix@C=1.3pc{\c{A} \ar[r]^{U} & \c{B} & \c{B} \ar[l]_-{\mathrm{id}_{\c{B}}}}$}\! is written $\com{U}{\c{B}}$.

\vspace*{1ex}

The notion and the theory of atoms only make sense in abelian categories. In general, the comma category is \textsl{not} abelian---not even if the categories $\c{A}$, $\c{B}$, and $\c{C}$ are abelian and the functors $U$ and $V$ are additive, as we have assumed. However, under weak assumptions, $\com{U}{V}$ \textsl{is} abelian, as we now prove. Two special cases of the following result can be found in \cite[Prop.~1.1 and remarks on p.~6]{FGR-75}, namely the cases where $U$ or $V$ is the identity functor.

\begin{prp}
  \label{prp:abelian}
  If $U$ is right exact and $V$ is left exact, then $\com{U}{V}$ is abelian.
\end{prp}

\begin{proof}
  It is straightforward to see that $\com{U}{V}$ is an additive category. 
  
  We now show that every morphism $(\alpha,\beta) \colon (A,B,\theta) \to (A',B',\theta')$ in $\com{U}{V}$ has a kernel. Let $\kappa \colon K \to A$ be a kernel of $\alpha$ and let $\lambda \colon L \to B$ be a kernel of $\beta$. As $V$ is left exact, the morphism $V\lambda \colon VL \to VB$ is a kernel of $V\beta$, so there is a (unique) morphism $\vartheta$ that makes the following diagram commute:
\begin{equation}
  \label{eq:kernel}
  \begin{gathered}
  \xymatrix{
    {} & UK \ar@{..}[d]\ar@{..}[d]\ar@{..}[d]\ar@{..>}[d]^-{\vartheta} \ar[r]^-{U\kappa} & UA \ar[r]^-{U\alpha} \ar[d]^-{\theta} & UA'\;\phantom{.} \ar@<-2pt>[d]^-{\theta'}
    \\
    0 \ar[r] & VL \ar[r]^-{V\lambda} & VB \ar[r]^-{V\beta} & VB'\;.
  }
  \end{gathered}  
\end{equation}  
We claim that $(\kappa,\lambda) \colon (K,L,\vartheta) \to (A,B,\theta)$ is a kernel of $(\alpha,\beta)$. By construction, the composition $(\alpha,\beta) \circ (\kappa,\lambda)$ is zero. Let $(\kappa',\lambda') \colon (K',L',\vartheta') \to (A,B,\theta)$ be any  morphism in $\com{U}{V}$ such that $(\alpha,\beta) \circ (\kappa',\lambda')$ is zero. We must show that $(\kappa',\lambda')$ factors uniquely through $(\kappa,\lambda)$. 

Note that we have unique factorizations \smash{$K' 
\stackrel{\text{\raisebox{-1pt}{\smash{$\varphi$}}}}{\longrightarrow} K \stackrel{\text{\raisebox{-1pt}{\smash{$\kappa$}}}}{\longrightarrow} A$} of $\kappa'$ and \smash{$L' 
\stackrel{\text{\raisebox{-1pt}{\smash{$\psi$}}}}{\longrightarrow} L \stackrel{\text{\raisebox{-1pt}{\smash{$\lambda$}}}}{\longrightarrow} B$} of $\lambda'$ by the universal property of kernels since $\alpha\kappa'=0$ and $\beta\lambda'=0$. From these factorizations, the commutativity of \eqref{kernel}, and from the fact that $(\kappa',\lambda')$ is a morphism in $\com{U}{V}$, we get:
\begin{displaymath}
  V\lambda \circ \vartheta \circ U\varphi \,=\, \theta \circ U\kappa \circ U\varphi \,=\, \theta \circ U\kappa' \,=\, V\lambda' \circ \vartheta' \,=\,  V\lambda \circ V\psi \circ \vartheta'\;.
\end{displaymath}
As $V\lambda$ is mono we conclude that $\vartheta \circ U\varphi = V\psi \circ \vartheta'$, so $(\varphi,\psi) \colon (K',L',\vartheta') \to (K,L,\vartheta)$ is a morphism in $\com{U}{V}$ with $(\kappa,\lambda) \circ (\varphi,\psi) = (\kappa',\lambda')$, that is, $(\kappa',\lambda')$ factors through $(\kappa,\lambda)$. 

A similar argument shows that every morphism in $\com{U}{V}$ has a cokernel; this uses the assumed right exactness of $U$. As for kernels, cokernels are computed componentwise.

Next we show that every monomorphism $(\alpha,\beta) \colon (A,B,\theta) \to (A',B',\theta')$ in $\com{U}{V}$ is a kernel. We have just shown that $(\alpha,\beta)$ has a kernel, namely $(K,L,\vartheta)$ where $K$ is a kernel of $\alpha$ and $L$ is a kernel of $\beta$. Thus, if $(\alpha,\beta)$ is mono, then $(K,L,\vartheta)$ is forced to be zero, so $\alpha$ and $\beta$ must both be mono. Let \smash{$0 \longrightarrow A 
\stackrel{\text{\raisebox{-1pt}{\smash{$\alpha$}}}}{\longrightarrow} A' \stackrel{\text{\raisebox{-1pt}{\smash{$\rho$}}}}{\longrightarrow} C \longrightarrow 0$} and 
\smash{$0 \longrightarrow B 
\stackrel{\text{\raisebox{-1pt}{\smash{$\beta$}}}}{\longrightarrow} B' \stackrel{\text{\raisebox{-1pt}{\smash{$\sigma$}}}}{\longrightarrow} D \longrightarrow 0$} be short exact sequences in $\c{A}$ and $\c{B}$. From the componentwise constructions of kernels and cokernels in $\com{U}{V}$ given above, it follows that $(\rho,\sigma)$ is a morphism in $\com{U}{V}$ whose kernel is precisely the given monomorphism $(\alpha,\beta)$.

A similar argument shows that every epimorphism in $\com{U}{V}$ is a cokernel.
\end{proof}

\begin{dfn}
  \label{dfn:stalk}
  As for quiver representations, see \dfnref{Si}, there are \emph{stalk functors},
\begin{displaymath}
  \xymatrix{
    \c{A} \ar[r]^-{S_{\mspace{-5mu}\c{A}}} & \com{U}{V} & 
    \c{B} \ar[l]_-{S_{\mspace{-4mu}\c{B}}}
  },
\end{displaymath}   
defined by $S_{\mspace{-5mu}\c{A}} \colon A \longmapsto (A,0,UA \xrightarrow{\ 0 \ } V0)$ and $S_{\mspace{-4mu}\c{B}} \colon B \longmapsto (0,B,U0 \xrightarrow{\ 0 \ } VB)$.
\end{dfn}

We now describe the right adjoints of these stalk functors.

\enlargethispage{3ex}

\begin{lem}
  \label{lem:K}
  The following asertions hold.
  \begin{prt}
    \item The stalk functor $S_{\mspace{-4mu}\c{B}}$ has a right adjoint $K_{\c{B}} \colon \com{U}{V} \to \c{B}$ given by $(X,Y,\theta) \mapsto Y$.

  \item Assume that $U$ has a right adjoint $U^!$ and let $\eta$ be the unit of the adjunction. The~stalk  functor $S_{\mspace{-5mu}\c{A}}$ has a right adjoint $K_{\c{A}} \colon \com{U}{V} \to \c{A}$ given by $(X,Y,\theta) \mapsto \Ker{(U^!\theta \circ \eta_X)}$, i.e.~the kernel of the morphism \smash{$X
\stackrel{\text{\raisebox{-0pt}{\smash{$\eta_X$}}}}{\longrightarrow} U^!UX \stackrel{\text{\raisebox{-1pt}{\smash{$U^!\theta$}}}}{\longrightarrow} U^!UY$}.
  \end{prt}
  
  In particular, if an object $(X,Y,\theta)$ in $\com{U}{V}$ satisfies $K_{\c{A}}(X,Y,\theta)=0$ and $K_{\c{B}}(X,Y,\theta)=0$, then one has $(X,Y,\theta)=0$.
\end{lem}

\begin{proof}
  \proofoftag{a} Let $B \in \c{B}$ and $(X,Y,\theta) \in \com{U}{V}$ be objects. It is immediate from \dfnref{stalk} that
  a morphism \mbox{$S_{\mspace{-4mu}\c{B}}(B) \to (X,Y,\theta)$} in $\com{U}{V}$ is the same as a morphism $\beta \colon B \to Y$ in $\c{B}$.

  \proofoftag{b} Write $\eta$ and $\varepsilon$ for the unit and counit of the assumed adjunction $(U,U^!)$. Let $A \in \c{A}$ and $(X,Y,\theta) \in \com{U}{V}$ be objects. It is immediate from \dfnref{stalk} that
  a morphism \mbox{$S_{\mspace{-5mu}\c{A}}(A) \to (X,Y,\theta)$} in $\com{U}{V}$ is the same as a morphism $\alpha \colon A \to X$ in $\c{A}$ such that the~composition $\theta \circ U\alpha \colon UA \to VY$ is zero. We claim that $\theta \circ U\alpha = 0$ if and only if $U^!\theta \circ \eta_X \circ \alpha = 0$. Indeed, the ``only if'' part follows directly from the identities
\begin{displaymath}
  U^!\theta \circ \eta_X \circ \alpha \,=\, U^!\theta \circ U^!U\alpha \circ \eta_A \,=\, U^!(\theta \circ U\alpha) \circ \eta_A\;,
\end{displaymath}  
where the first equality holds by naturality of $\eta$. The ``if'' part follows from the identities
\begin{displaymath}
  \theta \circ U\alpha \,=\, 
  \theta \circ \varepsilon_{UX} \circ U\eta_X \circ U\alpha \,=\,
  \varepsilon_{VY} \circ UU^!\theta \circ U\eta_X \circ U\alpha \,=\,  
  \varepsilon_{VY} \circ U(U^!\theta \circ \eta_X \circ \alpha)\;,
\end{displaymath} 
where the first equality is by the unit-counit relation \cite[IV.1 Thm.~1(ii)]{Mac} and the second is by naturality of $\varepsilon$. This is illustrated in the following diagram:
\begin{displaymath}
  \xymatrix{
    UA \ar[d]_-{U\alpha} \ar[r]^-{U\alpha} & UX \ar[r]^-{\theta} & VY
    \\
    UX \ar@{=}[ur] \ar[r]_-{U\eta_X} & 
    UU^!UX \ar[r]_-{UU^!\theta} \ar[u]_-{\varepsilon_{UX}}
    & UU^!VY\,. \ar[u]_-{\varepsilon_{VY}}
  }
\end{displaymath}
Therefore, a morphism $S_{\mspace{-5mu}\c{A}}(A) \to (X,Y,\theta)$ in $\com{U}{V}$ is the same as a morphism \mbox{$\alpha \colon A \to X$} in $\c{A}$ with $U^!\theta \circ \eta_X \circ \alpha = 0$, and by the universal property of the kernel, such morphisms are in bijective correspondance with morphisms $A \to \Ker{(U^!\theta \circ \eta_X)}$. This proves (b).

For the last statement, note that $K_{\c{B}}(X,Y,\theta)=0$ yields $Y=0$. Thus $\theta$ is the zero morphism $UX \to 0$ and consequently $U^!\theta \circ \eta_X$ is the zero morphism $X \to 0$. It follows that $X=K_{\c{A}}(X,Y,\theta)=0$, so $(X,Y,\theta)=0$ in $\com{U}{V}$.
\end{proof}

We are now in a position to show Theorem~B from the Introduction.

\begin{proof}[Proof of Theorem~B]
First notice that under the given assumptions, the comma category $\com{U}{V}$ is abelian by \prpref{abelian}, so it makes sense to consider its atom spectrum. We will apply \thmref{bijective} to the functors $S_{\mspace{-5mu}\c{A}}$ and $S_{\mspace{-4mu}\c{B}}$ from \dfnref{stalk} whose right adjoints are $K_{\c{A}}$ and $K_{\c{B}}$ from \lemref{K}. As shown in the proof of \prpref{abelian}, kernels and cokernels in $\com{U}{V}$ are computed componentwise, so the functor $S_{\mspace{-5mu}\c{A}}$ is exact. It also lifts subobjects as every subobject of $S_{\mspace{-5mu}\c{A}}(A)$ has the form $S_{\mspace{-5mu}\c{A}}(A')$ for a subobject \mbox{$A' \rightarrowtail A$} in $\c{A}$. It is clear from the definitions that the unit $\mathrm{id}_{\c{A}} \to K_{\c{A}}\mspace{1mu}S_{\mspace{-5mu}\c{A}}$ of the adjunction $(S_{\mspace{-5mu}\c{A}},K_{\c{A}})$ is an isomorphism, and hence $S_{\mspace{-5mu}\c{A}}$ is full and faithful by (the dual of) \cite[IV.3, Thm.~1]{Mac}. Similar arguments show that the functor $S_{\mspace{-4mu}\c{B}}$ has the same properties as those just established for~$S_{\mspace{-5mu}\c{A}}$. Therefore, the functors $S_{\mspace{-5mu}\c{A}}$ and $S_{\mspace{-4mu}\c{B}}$ meet the requirements in \prpref{f}.
 
 It remains to verify conditions (a) and (b) in \thmref{bijective}. However, condition (a) is straightforward from \dfnref{stalk}, and (b) holds by \lemref{K}. 
\end{proof}

\begin{exa}
  \label{exa:T-Mod}
  Let $A$ and $B$ be rings and let $M = {}_BM_A$ be a $(B,A)$-bimodule. We consider the comma category associated to $U= M \otimes_A - \colon \Mod{A} \to \Mod{B}$ and $V$ being the identity functor on $\Mod{B}$.
Theorem~B yields a homeomorphism and an order-isomorphism,
\begin{displaymath}  
  f \colon \ASpecp{\Mod{A}} \,\sqcup\, \ASpecp{\Mod{B}}
  \longrightarrow \ASpec{\com{(M \otimes_A -)}{(\Mod{B})}}\;,
\end{displaymath}    
which we now describe in more detail. There is a well-known equivalence of categories,
  \begin{displaymath}
    \com{(M \otimes_A -)}{(\Mod{B})} \stackrel{E}{\longrightarrow} \Mod{T}
    \qquad \textnormal{where} \qquad 
    T = 
    \begin{pmatrix}
    A & 0 \\
    M & B
    \end{pmatrix};
  \end{displaymath}
  see \cite{FGR-75} and \cite[Thm.~(0.2)]{Green}. Under this equivalence, an object $(X,Y,\theta)$ in the comma category is mapped to the left $T$-module whose underlying abelian group is $X \oplus Y$ where $T$-multiplication is defined by
\begin{displaymath}
  \begin{pmatrix}
    a & 0 \\
    m & b
  \end{pmatrix}\!\!
  \begin{pmatrix}
    x \\
    y
  \end{pmatrix}
  =
  \begin{pmatrix}
    ax \\
    \theta(m \otimes x)+by
  \end{pmatrix}  
  \qquad \textnormal{for} \qquad
    \begin{pmatrix}
    a & 0 \\
    m & b
  \end{pmatrix} \in T
  \ \ \textnormal{ and } \ \
  \begin{pmatrix}
    x \\
    y
  \end{pmatrix}
  \in
  {\renewcommand{\arraystretch}{0.7} 
  \begin{array}{c}
  X \\ \oplus \vspace{1.5pt} \\ Y
  \end{array}
  }.  
\end{displaymath}

For simplicity we now consider the case where $A$ and $B$ are commutative (but $T$ is not). Define a map $\tilde{f}$ by commutativity of the diagram
\begin{displaymath}
  \xymatrix{
    \Spec{A} \,\sqcup\, \Spec{B} \ar@{..}[r]\ar@{..}[r]\ar@{..}[r]\ar@{..>}[r]^-{\tilde{f}}
    \ar[d]_-{q_A\,\sqcup\,\, q_B}^-{\sim} & 
    \ASpecp{\Mod{T}} 
    \\
    \ASpecp{\Mod{A}} \,\sqcup\, \ASpecp{\Mod{B}} \ar[r]^-{f} & 
    \ASpec{\com{(M \otimes_A -)}{(\Mod{B})}}\;, \ar[u]_-{\ASpec{E}}^-{\sim}
  }
\end{displaymath} 
where $q_A$ and $q_B$ are the homeomorphisms and order-isomorphisms from \ref{Spec}. By using the definitions of these maps, it follows easily that
\begin{displaymath}
  \tilde{f}(\mathfrak{p}) \,=\, \aa
  {T\!\!\bigg/\!\!
  \begin{pmatrix}
    \mathfrak{p} & 0 \\
    M & B
  \end{pmatrix}}
  \qquad \textnormal{and} \qquad
  \tilde{f}(\mathfrak{q}) \,=\, \aa{T\!\!\bigg/\!\!
  \begin{pmatrix}
    A & 0 \\
    M & \mathfrak{q}
  \end{pmatrix}}
\end{displaymath}  
for $\mathfrak{p} \in \Spec{A}$ and $\mathfrak{q} \in \Spec{B}$. In the terminology of \cite[Def.~6.1]{Kanda-Serre} the denominators above are \emph{comonoform} left ideals in $T$. For $A=B=M=K$, a field, this recovers \cite[Exa.~8.3]{Kanda-Serre}\footnote{\,This example, which inspired the present paper, was worked out using methods different from what we have developed here. The approach in \cite[Exa.~8.3]{Kanda-Serre} is that one can write down \emph{all} ideals in a lower triangular matrix ring, see for example \cite[Prop.~(1.17)]{Lam-AFCINR}, and from this list it is possible to single out the comonoform ones.}. For $A=B=M=\Bbbk$, where $\Bbbk$ is any commutative ring, the conclusion above also follows from   \exaref{subspace} with $n=2$.
\end{exa}

\appendix

\section{Quivers with relations and their representations}
\label{app:Quiver}

In this appendix, we present some (more or less standard) background material on representations of quivers with relations that we will need, and take for granted, in \secref{Application-quiver}.

\begin{ipg}
  \label{app-quiver}
A \emph{quiver} is a directed graph. For a quiver $Q$ we denote by $Q_0$ the set of vertices and~by $Q_1$ the set of arrows in $Q$. Unless otherwise specified there are no restrictions on a quiver; it may have infinitely many vertices, it may have loops and/or~oriented cycles, and there may be infinitely many or no arrows from one vertex to another.

 For an arrow \mbox{$a \colon\! i \to j$} in $Q$ the vertex $i$, respectively, $j$, is called the \emph{source}, respectively, \emph{target}, of $a$. A \emph{path} $p$ in $Q$ is a finite sequence of composable arrows \smash{$\bullet \stackrel{\text{\raisebox{1ex}{$a_1$}}}{\smash{\longrightarrow}} \bullet \stackrel{\text{\raisebox{1ex}{$a_2$}}}{\smash{\longrightarrow}} \cdots \stackrel{\text{\raisebox{1ex}{$a_n$}}}{\smash{\longrightarrow}} \bullet$} (that is, the target of $a_\ell$ equals the source of $a_{\ell+1}$), which we write $p=a_n\cdots a_2a_1$. If $p$~and~$q$ are paths in $Q$ and the target of $q$ coincides with the source of $p$, then we write $pq$ for the composite path (i.e.~first $q$, then $p$). At each vertex $i \in Q_0$ there is by definition a \emph{trivial path}, denoted by $e_i$, whose source and target are both $i$. For every path $p$ in $Q$ with source~$i$ and target $j$ one has $pe_i=p=e_{\!j}p$.

  Let $Q$ be a quiver and let $\c{A}$ be an abelian category. One can view $Q$ as a category, which we denote by $\bar{Q}$, whose objects are vertices in $Q$ and whose morphisms are paths in $Q$. An \emph{$\c{A}$-valued representation of $Q$} is a functor $X \colon \bar{Q} \to \c{A}$ and a morphism $\lambda \colon X \to Y$ of representations $X$ and $Y$ is a natural transformation. The category of $\c{A}$-valued representations of $Q$, that is, the category of functors $\bar{Q} \to \c{A}$, is written $\Rep{Q}{\c{A}}$. In symbols:
\begin{equation}
  \label{eq:Rep-def}
  \Rep{Q}{\c{A}} \,=\, \mathrm{Func}(\bar{Q},\c{A}).
\end{equation} 
  It is an abelian category where kernels and cokernels 
are computed vertexwise.
\end{ipg}  

\begin{ipg}
  \label{app-k-linear}
  Let $\Bbbk$ be a commutative ring. Recall that a \emph{$\Bbbk$-linear category} is a category $\c{K}$ enriched in the monoidal category $\Mod{\Bbbk}$ of $\Bbbk$-modules, that is, the hom-sets in $\c{K}$ have structures of $\Bbbk$-modules and composition in $\c{K}$ is $\Bbbk$-bilinear. If $\c{K}$ and $\c{L}$ are $\Bbbk$-linear categories, then we write $\mathrm{Func}_{\Bbbk}(\c{K},\c{L})$ for the category of $\Bbbk$-linear functors from $\c{K}$ to $\c{L}$. Here we must require that $\c{K}$ is skeletally small in order for $\mathrm{Func}_{\Bbbk}(\c{K},\c{L})$ to have small hom-sets.
  
    If $\c{C}$ is any category we write $\Bbbk\c{C}$ for the category whose objects are the same as those in $\c{C}$ and where $\Hom{\Bbbk\c{C}}{X}{Y}$ is the free $\Bbbk$-module on the set $\Hom{\c{C}}{X}{Y}$. Composition in $\Bbbk\c{C}$ is induced by composition in $\c{C}$. The category $\Bbbk\c{C}$ is evidently $\Bbbk$-linear and we call it the \emph{$\Bbbk$-linearization} of $\c{C}$. Note that there is canonical functor $\c{C} \to \Bbbk\c{C}$. For any skeletally small category $\c{C}$ and any $\Bbbk$-linear category $\c{L}$ there is an equivalence of categories,
\begin{equation}
  \label{eq:Func}
  \mathrm{Func}(\c{C},\c{L}) \,\simeq\, \mathrm{Func}_{\Bbbk}(\Bbbk\c{C},\c{L})\;.
\end{equation} 
That is, (ordinary) functors $\c{C} \to \c{L}$ corrspond to $\Bbbk$-linear functors $\Bbbk\c{C} \to \c{L}$. This equivalence maps a functor $F \colon \c{C} \to \c{L}$ to the $\Bbbk$-linear functor $\tilde{F} \colon \Bbbk\c{C} \to \c{L}$ given by $\tilde{F}(C) = F(C)$ for any object $C$ and $\tilde{F}(x_1\varphi_1+\cdots+x_m\varphi_m) = x_1F(\varphi_1)+\cdots+x_mF(\varphi_m)$ for any morphism $x_1\varphi_1+\cdots+x_m\varphi_m$ in $\Bbbk\c{C}$ (where $x_u \in \Bbbk$ and $\varphi_1,\ldots,\varphi_m \colon C \to C'$ are morphisms in $\c{C}$). In the other direction, \eqref{Func} maps
a $\Bbbk$-linear functor $\Bbbk\c{C} \to \c{L}$ to the composition
$\c{C} \to \Bbbk\c{C} \to \c{L}$.

A \emph{two-sided ideal} $\c{I}$ in a $\Bbbk$-linear category $\c{K}$ is a collection of $\Bbbk$-submodules $\c{I}(X,Y) \subseteq \Hom{\c{K}}{X}{Y}$, indexed by pairs $(X,Y)$ of objects in $\c{K}$, such that 
\begin{itemlist}
\item For every $\beta \in \Hom{\c{K}}{Y}{Y'}$ and $\varphi \in \c{I}(X,Y)$ one has $\beta \varphi \in \c{I}(X,Y')$, and
\item For every $\varphi \in \c{I}(X,Y)$ and $\alpha \in \Hom{\c{K}}{X'}{X}$ one has $\varphi\alpha \in \c{I}(X',Y)$.
\end{itemlist}
Given such an ideal $\c{I}$ in $\c{K}$ one can define the quotient category $\c{K}/\c{I}$, which has the same objects as $\c{K}$ and hom-sets defined by (quotient of $\Bbbk$-modules):
\begin{displaymath}
  \Hom{\c{K}/\c{I}}{X}{Y} \,=\, \Hom{\c{K}}{X}{Y}/\c{I}(X,Y)\;.
\end{displaymath} 
Composition in $\c{K}/\c{I}$ is induced from composition in $\c{K}$, and it is well-defined since $\c{I}$ is a two-sided ideal. It is straightforward to verify the $\c{K}/\c{I}$ is a $\Bbbk$-linear category. There is a canonical $\Bbbk$-linear functor $\c{K} \to \c{K}/\c{I}$, which for any $\Bbbk$-linear category $\c{L}$ induces a functor $\mathrm{Func}_{\Bbbk}(\c{K}/\c{I},\c{L}) \to \mathrm{Func}_{\Bbbk}(\c{K},\c{L})$. It is not hard to see that this functor is fully faithful, so $\mathrm{Func}_{\Bbbk}(\c{K}/\c{I},\c{L})$ may be identified with a full subcategory of $\mathrm{Func}_{\Bbbk}(\c{K},\c{L})$. In fact, one has
\begin{displaymath}
  \mathrm{Func}_{\Bbbk}(\c{K}/\c{I},\c{L}) \,\simeq\,
  \{F \in \mathrm{Func}_{\Bbbk}(\c{K},\c{L}) \mspace{3mu}|\mspace{3mu} \textnormal{$F$ kills $\c{I}$} \mspace{2mu} \}\;.
\end{displaymath}

If $\c{R}$ is a collection of morphisms in a $\Bbbk$-linear category $\c{K}$, then we write $(\c{R})$ for the two-sided ideal in $\c{K}$ generated by $\c{R}$. I.e.~$(\c{R})(X,Y)$ consists of finite sums $\sum_{u} x_u\mspace{2mu}\beta_u \varphi_u \alpha_u$ where $x_u \in \Bbbk$ and $\alpha_u \colon X \to X_u$,\, $\varphi_u \colon X_u \to Y_u$,\, $\beta_u \colon Y_u \to Y$ are morphisms in $\c{K}$ with $\varphi_u \in \c{R}$.
\end{ipg}

\begin{ipg}
  \label{app-relations}
Let $Q$ be a quiver and let $\Bbbk$ be a commutative ring. Consider the $\Bbbk$-linear category~$\Bbbk \bar{Q}$, that is, the $\Bbbk$-linearization (see \ref{app-k-linear}) of the category $\bar{Q}$ (see \ref{app-quiver}). 

A \emph{relation} (more precisely, a \emph{$\Bbbk$-linear relation})  in $Q$ is a morphism $\rho$ in $\Bbbk \bar{Q}$. That is, 
$\rho$ is a formal $\Bbbk$-linear combination $\rho = x_1p_1+\cdots+x_mp_m$ ($x_u \in \Bbbk$) of paths $p_1,\ldots,p_m$ in $Q$ with a common source and a common taget. 

A \emph{quiver with relations} is a pair $(Q,\c{R})$ with $Q$ a quiver and $\c{R}$ a set of relations in $Q$.

  Let $\c{A}$ be a $\Bbbk$-linear abelian category. For a representation $X \in \Rep{Q}{\c{A}}$, as in \ref{app-quiver}, and a relation $\rho=x_1p_1+\cdots+x_mp_m$ in $Q$, define $X(\rho):=x_1X(p_1)+\cdots+x_mX(p_m)$. One says that $X$ \emph{satisfies} the relation $\rho$ if $X(\rho)=0$.
  
  If $(Q,\c{R})$ is a quiver with relations, then an \emph{$\c{A}$-valued representation of $(Q,\c{R})$}~is~a~repre\-sentation $X \in \Rep{Q}{\c{A}}$ with $X(\rho)=0$ for all $\rho \in \c{R}$,  that is, $X$ satisfies all relations~in~$\c{R}$. Write $\Rep{(Q,\c{R})}{\c{A}}$ for the category of $\c{A}$-valued representations of $(Q,\c{R})$. In symbols:
\begin{displaymath}
  \Rep{(Q,\c{R})}{\c{A}} \,=\, \{X \in \Rep{Q}{\c{A}} \mspace{3mu}|\mspace{3mu} X(\rho)=0 \textnormal{ for all } \rho \in \c{R} \}\;.
\end{displaymath}  
We consider $\Rep{(Q,\c{R})}{\c{A}}$ as a full subcategory of $\Rep{Q}{\c{A}}$.
 We have a diagram:
\begin{displaymath}
  \xymatrix{
    \Rep{Q}{\c{A}} \ar[r]^-{\simeq} &
    \mathrm{Func}_{\Bbbk}(\Bbbk\bar{Q},\c{A})
    \\
    \Rep{(Q,\c{R})}{\c{A}} \ar@{>->}[u]
    \ar@{..}[r]\ar@{..}[r]\ar@{..}[r]\ar@{..>}[r]^-{\simeq}
    &
    \mathrm{Func}_{\Bbbk}(\Bbbk\bar{Q}/(\c{R}),\c{A})\;,
    \ar@{>->}[u]
  }
\end{displaymath}
where the upper horizontal equivalence comes from \eqref{Rep-def} and \eqref{Func}. The vertical functors are inclusions. It is immediate from the definitions that the equivalence in the top row restricts to an equivalence in the bottom row, so we get commutativity of the displayed diagram.
\end{ipg}

\begin{ipg}
  \label{app-path-algebra}
Let $Q$ be a quiver with finitely many vertices(!) and let $\Bbbk$ be a commutative ring. The \emph{path algebra} $\Bbbk Q$ is the $\Bbbk$-al\-ge\-bra whose underlying $\Bbbk$-module is free with basis all paths in $Q$ and multiplication of paths $p$ and $q$ are given by their composition $pq$, as in \ref{app-quiver}, if they are composable, and $pq=0$ if they are not composable. Note that $\Bbbk Q$ has unit $\sum_{i \in Q_0}e_i$.
  
  There is an equivalence of categories, see e.g.~\cite[Lem.~p.~6]{CBquiver} or \cite[Chap.~III.1 Thm.~1.6]{ASS1}:
  \begin{equation}
    \label{eq:UV}
      \Rep{Q}{\Mod{\Bbbk}} \,\simeq \Mod{\Bbbk Q}\;.
  \end{equation}  
  We describe the quasi-inverse functors $U$ and $V$ that give this equivalence. A representation $X$ is mapped to the left $\Bbbk Q$-module $UX$ whose underlying $\Bbbk$-mo\-dule is $\bigoplus_{i \in Q_0}X(i)$; multiplication by paths works as follows: Let \mbox{$\varepsilon_i \colon X(i) \rightarrowtail \bigoplus_{i \in Q_0}X(i)$} and \mbox{$\pi_i \colon \bigoplus_{i \in Q_0}X(i) \twoheadrightarrow X(i)$} be the \smash{$i^\mathrm{th}$} injection and projection in $\Mod{\Bbbk}$. For a path $p \colon \mspace{-2mu}i \rightsquigarrow \mspace{-2mu}j$ and an element $z \in UX$ one has $pz=(\varepsilon_{\mspace{-2mu}j} \circ X(p) \circ \pi_i)(z)$. In the other direction, a left $\Bbbk Q$-module $M$ is mapped to the representation $VM$ given by $(VM)(i)=e_iM$ for $i \in Q_0$. For a path $p \colon \mspace{-2mu}i \rightsquigarrow \mspace{-2mu}j$ in $Q$ the $\Bbbk$-homomorphism $(VM)(p) \colon e_iM \to e_{\mspace{-2mu}j}M$ is left multiplication by $p$.
 
By definition, see \ref{app-relations}, a relation in $Q$ can be viewed as an element (of a special kind) in the algebra $\Bbbk Q$. If $(Q,\c{R})$ is a quiver with relations and $I=(\c{R})$ is the two-sided ideal in $\Bbbk Q$ generated by the subset $\c{R} \subseteq \Bbbk Q$, then we have a diagram:
\begin{displaymath}
  \xymatrix{
    \Rep{Q}{\Mod{\Bbbk}} \ar[r]^-{\simeq} &
    \Mod{\Bbbk Q}
    \\
    \Rep{(Q,\c{R})}{\Mod{\Bbbk}} \ar@{>->}[u]
    \ar@{..}[r]\ar@{..}[r]\ar@{..}[r]\ar@{..>}[r]^-{\simeq}
    &
    \Mod{\Bbbk Q/I}\;,
    \ar@{>->}[u]
  }
\end{displaymath}
where the upper horizontal equivalence is \eqref{UV}. The vertical functors are inclusions, where $\Mod{\Bbbk Q/I}$ is identified with the full subcategory $\{M \in \Mod{\Bbbk Q} \mspace{3mu}|\mspace{3mu} IM=0 \}$ of $\Mod{\Bbbk Q}$. It is immediate from the definitions that the equivalence in the top row restricts to an equivalence in the bottom row, so we get commutativity of the displayed diagram.
\end{ipg}

\section*{Acknowledgement}

We thank Ryo Kanda for fruitful and interesting discussions about this work. We are grateful to the anonymous referees for valuable comments that substantially improved the manuscript and corrected some mistakes.

\def\cprime{$'$}
  \providecommand{\arxiv}[2][AC]{\mbox{\href{http://arxiv.org/abs/#2}{\sf
  arXiv:#2 [math.#1]}}}
  \providecommand{\oldarxiv}[2][AC]{\mbox{\href{http://arxiv.org/abs/math/#2}{\sf
  arXiv:math/#2
  [math.#1]}}}\providecommand{\MR}[1]{\mbox{\href{http://www.ams.org/mathscinet-getitem?mr=#1}{#1}}}
  \renewcommand{\MR}[1]{\mbox{\href{http://www.ams.org/mathscinet-getitem?mr=#1}{#1}}}
\providecommand{\bysame}{\leavevmode\hbox to3em{\hrulefill}\thinspace}
\providecommand{\MR}{\relax\ifhmode\unskip\space\fi MR }
\providecommand{\MRhref}[2]{%
  \href{http://www.ams.org/mathscinet-getitem?mr=#1}{#2}
}
\providecommand{\href}[2]{#2}


\begin{thebibliography}{10}

\bibitem{ASS1}
Ibrahim Assem, Daniel Simson, and Andrzej Skowro{\'n}ski, \emph{Elements of the
  representation theory of associative algebras. {V}ol. 1}, London Math. Soc.
  Stud. Texts, vol.~65, Cambridge University Press, Cambridge, 2006, Techniques
  of representation theory. \MR{MR2197389}

\bibitem{CBquiver}
William Crawley-Boevey, \emph{Lectures on representations of quivers}, 1992,
  Mathematical Institute, Oxford University, available from
  https://www.math.uni-bielefeld.de/$\sim$wcrawley/.

\bibitem{Diers}
Yves Diers, \emph{Categories of {B}oolean sheaves of simple algebras}, Lecture
  Notes in Math., vol. 1187, Springer-Verlag, Berlin, 1986. \MR{MR841523}

\bibitem{EEGR09}
Edgar~E. Enochs, Sergio Estrada, and Juan~Ram{\'o}n Garc{\'{\i}}a~Rozas,
  \emph{Injective representations of infinite quivers. {A}pplications}, Canad.
  J. Math. \textbf{61} (2009), no.~2, 315--335. \MR{MR2504018}

\bibitem{FGR-75}
Robert~M. Fossum, Phillip~A. Griffith, and Idun Reiten, \emph{Trivial
  extensions of abelian categories}, Lecture Notes in Math., vol. 456,
  Springer-Verlag, Berlin-New York, 1975, Homological algebra of trivial
  extensions of abelian categories with applications to ring theory.
  \MR{MR0389981}

\bibitem{PGb62}
Pierre Gabriel, \emph{Des cat\'egories ab\'eliennes}, Bull. Soc. Math. France
  \textbf{90} (1962), 323--448. \MR{MR0232821}

\bibitem{Green}
Edward~L. Green, \emph{On the representation theory of rings in matrix form},
  Pacific J. Math. \textbf{100} (1982), no.~1, 123--138. \MR{MR661444}

\bibitem{Hochster}
Melvin Hochster, \emph{Prime ideal structure in commutative rings}, Trans.
  Amer. Math. Soc. \textbf{142} (1969), 43--60. \MR{MR0251026}

\bibitem{HJ-quiver}
Henrik Holm and Peter J{\o}rgensen, \emph{Cotorsion pairs in categories of
  quiver representations}, to appear in Kyoto J. Math., 23 pp.,
  \arxiv[CT]{1604.01517v2}.

\bibitem{Kanda-Serre}
Ryo Kanda, \emph{Classifying {S}erre subcategories via atom spectrum}, Adv.
  Math. \textbf{231} (2012), no.~3-4, 1572--1588. \MR{MR2964615}

\bibitem{Kanda-Extension}
\bysame, \emph{Extension groups between atoms and objects in locally noetherian
  {G}rothendieck category}, J. Algebra \textbf{422} (2015), 53--77.
  \MR{MR3272068}

\bibitem{Kanda-Specialization-orders}
\bysame, \emph{Specialization orders on atom spectra of {G}rothendieck
  categories}, J. Pure Appl. Algebra \textbf{219} (2015), no.~11, 4907--4952.
  \MR{MR3351569}

\bibitem{Lam-AFCINR}
Tsit~Y. Lam, \emph{A first course in noncommutative rings}, second ed., Grad.
  Texts in Math., vol. 131, Springer-Verlag, New York, 2001. \MR{MR1838439}

\bibitem{Mac}
Saunders Mac~Lane, \emph{Categories for the working mathematician}, Grad. Texts
  in Math., vol.~5, Springer-Verlag, New York, 1971. \MR{MR0354798}

\bibitem{EMt58}
Eben Matlis, \emph{Injective modules over {N}oetherian rings}, Pacific J. Math.
  \textbf{8} (1958), 511--528. \MR{MR0099360}

\bibitem{Sto72}
Hans~H. Storrer, \emph{On {G}oldman's primary decomposition},  (1972),
  617--661. Lecture Notes in Math., vol. 246. \MR{MR0360717}

\end{thebibliography}
\end{document}